\documentclass[11pt]{amsart}
    \pdfoutput=1
\usepackage[maxbibnames=99]{biblatex}
    \usepackage{graphicx}
    \usepackage[english]{babel}
    \usepackage{graphicx}
    \usepackage{framed}
    \usepackage[normalem]{ulem}
    \usepackage{amsmath}
    \usepackage{dynkin-diagrams}
    \usepackage{amsthm}
    \usepackage{amssymb}
    \usepackage{hyperref} 
    \usepackage[capitalize]{cleveref}
    \usepackage{comment}
    \usepackage{tikz}
    \usetikzlibrary{matrix,calc}
    \usepackage{mathtools}
    \usepackage{tikz-cd}
    \usepackage{amsfonts}
    \usepackage{enumerate}
    \usepackage[utf8]{inputenc}
    \usepackage[top=1 in,bottom=1in, left=1 in, right=1 in]{geometry}
       \usepackage{stmaryrd}

    \usepackage[colorinlistoftodos]{todonotes}
    \usepackage[all]{xy}
    \usepackage{mathrsfs}

    \newcommand{\G}{\mathbb{G}} 
    
    \newcommand{\LSL}{\mathfrak{sl}}
    \newcommand{\SL}{\operatorname{SL}}

    \newcommand{\LG}{\mathfrak{g}}
    
    \newcommand{\LT}{\mathfrak{t}}
    \newcommand{\LN}{\mathfrak{n}}
    \newcommand{\LB}{\mathfrak{b}}

    \newcommand{\QCoh}{\operatorname{QCoh}}
    \newcommand{\IndCoh}{\operatorname{IndCoh}}
    
    \newcommand{\C}{\mathcal{C}} 
    \newcommand{\D}{\mathcal{D}} 
    \renewcommand{\\}{\backslash}

    \theoremstyle{definition}
    \newtheorem{Theorem}{Theorem}[section]
    
    \newtheorem{Question}[Theorem]{Question}
    \newtheorem{Corollary}[Theorem]{Corollary}
    \newtheorem{Definition}[Theorem]{Definition}
    \newtheorem{Proposition}[Theorem]{Proposition}
    
    \newtheorem{Remark}[Theorem]{Remark}
    
    \newtheorem{Lemma}[Theorem]{Lemma}

    \title{Proof of the Ginzburg-Kazhdan Conjecture}
    \author{Tom Gannon}

    \newcommand{\AvN}{\text{Av}_*^N}
    \newcommand{\Avpsi}{\text{Av}_{!}^{\psi}}
    \newcommand{\Symt}{\text{Sym(}\mathfrak{t}\text{)}}

    \newcommand{\LTd}{\LT^{\ast}}

\begin{document}
    
\renewcommand{\G}{\mathcal{G}}
\newcommand{\AvNTw}{\text{Av}_*^{N, (T, w)}}
\newcommand{\AvGw}{\text{Av}_{\ast}^{G,w}}
\newcommand{\pifin}{\pi_{\text{fin}}}
\newcommand{\pifinL}{\pi_{\text{fin,L}}}

    \newcommand{\ELeftAdjoint}{\text{ev}_{\omega_{\LTd}}}
\newcommand{\ClGlobalDiffOp}{\text{H}^0\Gamma(\mathcal{D}_{G/N})}
\newcommand{\GlobalDiffOp}{\Gamma(\mathcal{D}_{G/N})}
\newcommand{\indsch}{\mathcal{X}}

    \newcommand{\DGCatContk}{\text{DGCat}^k_{\text{cont}}}
\newcommand{\DGCatContL}{\text{DGCat}^L_{\text{cont}}}

\newcommand{\DNTw}{\mathcal{D}(N\backslash G/N)^{T_r,w}}
\newcommand{\DNTwWhit}{\mathcal{D}(N^-_{\psi}\backslash G/N)^{T_r,w}}
\newcommand{\DNWhit}{\mathcal{D}(N^-_{\psi}\backslash G/N)}
\newcommand{\DNTwldeg}{\mathcal{D}(N \backslash G/N)^{T_r,w}_{\text{left-deg}}}
\newcommand{\DNTwnondeg}{\mathcal{D}(N \backslash G/N)^{T_r,w}_{\text{nondeg}}}
\newcommand{\DN}{\mathcal{D}(N\backslash G/N)}
\newcommand{\DNldeg}{\mathcal{D}(N \backslash G/N)_{\text{left-deg}}}
\newcommand{\DNnondeg}{\mathcal{D}(N \backslash G/N)_{\text{nondeg}}}
\newcommand{\DNlambda}{\mathcal{D}^{\lambda}(N \backslash G/B)}
\newcommand{\Dpsilambda}{\mathcal{D}^{\lambda}(N^- _{\psi}\backslash G/B)}
\newcommand{\DbiTw}{\mathcal{D}(N \backslash G/N)^{T \times T, w}}
\newcommand{\DbiTwnondeg}{\mathcal{D}(N \backslash G/N)^{T \times T, w}_{\text{nondeg}}}
\newcommand{\DbiTwnondegheart}{\mathcal{D}(N \backslash G/N)^{(T \times T, w), \heartsuit}_{\text{nondeg}}}
\newcommand{\DbiTwdeg}{\mathcal{D}(N \backslash G/N)^{T \times T, w}_{\text{deg}}}
\newcommand{\DNBlambda}{\mathcal{D}(N \backslash G/_{\lambda}B)}
\newcommand{\DNTwBlambda}{\mathcal{D}(N \backslash G/_{\lambda}B)}
\newcommand{\DWhitBlambda}{\mathcal{D}(N^-_{\psi}\backslash G/_{\lambda}B)}
\newcommand{\HN}{\D(N \backslash G/N)}
\newcommand{\HNTw}{\D(N \backslash G/N)^{T \times T, w}}
\newcommand{\HNTwabbreviated}{\mathcal{H}^{N, (T,w)}}
\renewcommand{\indsch}{\mathcal{X}}
\newcommand{\wdot}{\dot{w}}
\newcommand{\Gaminusalpha}{\mathbb{G}_a^{-\alpha}}
\newcommand{\Aone}{\mathbf{A}}
\newcommand{\Atwo}{\mathcal{L}\text{-mod}(\Aone)}
\newcommand{\algobj}{\mathcal{A}}
\newcommand{\newalgobj}{\mathcal{A}'}
\newcommand{\tilder}{\tilde{r}}
\newcommand{\Ccirc}{\mathring{\C}}
\newcommand{\rootlattice}{\mathbb{Z}\Phi}
\newcommand{\characterlatticeforT}{X^{\bullet}(T)}
\newcommand{\Spec}{\text{Spec}}
\newcommand{\FourierMukai}{\text{FMuk}}
\newcommand{\oneshiftedCartierdual}{c_1}
\newcommand{\quotientmapforcoarsequotient}{\overline{s}}

\newcommand{\tildeV}{\tilde{\mathbb{V}}}
\newcommand{\Vdual}{V^{\vee}}
\newcommand{\generalstacktoGITquotientmap}{\phi}
\newcommand{\SpecofL}{\text{Spec}(L)}
\newcommand{\terminalmapfromC}{\alpha}
\newcommand{\terminalmapfromCmodassociatedstabilizer}{\dot{\alpha}}
\newcommand{\Hpsiliteral}{\mathcal{H}_{\psi}}
\newcommand{\AvNshifted}{\AvN[\text{dim}(N)]}
\newcommand{\hyperplanefixedbys}{V^{\ast}_{s = \text{id}}}
\newcommand{\fieldpossiblydifferentfromgroundfield}{K}
\newcommand{\Avpsishifted}{\Avpsi[-\text{dim}(N)]}
\newcommand{\FI}{F_{I}}
\newcommand{\FD}{F_{\mathcal{D}}}
\newcommand{\LI}{s_*^{\IndCoh}}
\newcommand{\LD}{\Avpsi}
\newcommand{\LIenh}{\pi_*^{\IndCoh, \text{enh}}}
\newcommand{\LDenh}{\text{Av}_!^{\psi, \text{enh}}}
\newcommand{\Ind}{\text{Ind}}
\newcommand{\Res}{\text{Res}}
\newcommand{\IndWithoutSignRep}{\text{Ind}(-)^W}
\newcommand{\ResWithoutSignRep}{\text{WRes}}
\newcommand{\IndWithSignRep}{\text{Ind}(- \otimes k_{\text{sign}})^W}
\newcommand{\ResWithSignRep}{\text{WRes}_s}
\newcommand{\universalhorocyclefunctor}{\widetilde{\text{hc}}}
\newcommand{\universalparabolicrestriction}{\widetilde{\text{Res}}}
\newcommand{\universalLIFTEDparabolicrestriction}{\widetilde{\text{WRes}}}
\newcommand{\parabolicrestrictionLIFTED}{\text{WRes}}
\newcommand{\horocycleFunctor}{\text{hc}}
\newcommand{\nondegenerateHorocycleFunctor}{\text{hc}_{\text{n}}}
\newcommand{\nondegenerateUniversalHorocycleFunctor}{\universalhorocyclefunctor_{\text{n}}}
\newcommand{\V}{\mathcal{V}}
\newcommand{\parabolicrestriction}{\text{Res}}

\newcommand{\conjugacyclassofstandardLevis}{\underline{\Theta}}
\newcommand{\moduliOfEvenMonicPolynomialsOfTHISDEGREE}[1]{\mathcal{M}^{#1}}
\newcommand{\moduliOfALLEvenPolynomialsOfTHISDEGREE}[1]{\mathcal{P}^{#1}}
\newcommand{\moduliOfEvenMonicPolynomialsOfdegreeTWOwithMINUSSIGN}{\mathcal{M}_{-}^{2}}

\newcommand{\basedQuasimapsfromProjectiveLinetoAffineClosureofSLTWOwithIsoClassofTHISDEGREE}[1]{\text{Maps}_*^{#1}(\mathbb{P}^1, \overline{\SL_2/N}/T)}
\newcommand{\basedQuasimapsfromProjectiveLinetoAffineClosureofSLTWOwithISOtoTHISDEGREE}[1]{\text{Maps}_*^{#1, \simeq}(\mathbb{P}^1, \overline{\SL_2/N}/T)}
\newcommand{\oblv}{\text{oblv}}
\newcommand{\explicitIsoofBasedQMapsforSLTWOwithPolynomialsforMapsofTHISDEGREE}[1]{\Phi_{#1}}
\newcommand{\basedQuasimapsfromProjectiveLinetoAffineClosureofBORELwithIsoClassofTHISDEGREE}[1]{\text{Maps}_*^{#1}(\mathbb{P}^1, \overline{B/N}/T)}

\newcommand{\basedmapsFromProjectiveLinetoTHISSPACE}[1]{\text{Maps}_*(\mathbb{P}^1, #1)}

\newcommand{\resolutionOfSingularitiesSpaceForBasicAffineSpace}{G\mathop{\times}\limits^{B} E}

\newcommand{\affineClosureofBasicAffineSpace}{\overline{G/N}}
\newcommand{\affineClosureOfCotangentBundleofBasicAffineSpace}{\overline{T^*(G/N)}}
\newcommand{\SpecOfSymofSectionsOfTangentBundleOfBasicAffineSpace}{\text{Spec}(\text{Sym}_{\affineClosureofBasicAffineSpace}^{\bullet}(\mathcal{T}_{\overline{G/N}}))}
\newcommand{\ringOfFunctionsForBasicAffineSpace}{A}
\newcommand{\ringOfFunctionsForCOTANGENTBUNDLEOfBasicAffineSpace}{R}
\newcommand{\tangentSheafForBasicAffineSpace}{\mathcal{T}_{G/N}}
\newcommand{\symOfDirectSumOfRepsofFundamentalWeights}{\text{Sym}(\oplus_i E(\omega_i))}
\newcommand{\Sym}{\text{Sym}}
\newcommand{\projectionFromAffineClosureofCotangentBundleToAffineClosureofSpace}{\overline{\pi}}
\newcommand{\momentMapFromAFFINECLOSUREofCotangentSpaceWithGROUPG}{\overline{\mu}_G}
\newcommand{\momentMapFromAFFINECLOSUREofCotangentSpaceWithGROUPT}{\overline{\mu}_T}
\newcommand{\momentMapFromAFFINECLOSUREofCotangentSpacewithTOTALGROUP}{\overline{\mu}}
\newcommand{\singularLocusOfAffineClosureOfCOTANGENTBUNDLEofBasicAffineSpace}{Z}
\newcommand{\singularLocusOfAffineClosureofBasicAffineSpace}{Z_0}
\newcommand{\idealForSingularLocusOfAffineClosureOfCOTANGENTBUNDLEofBasicAffineSpace}{I_Z}
\newcommand{\groundfield}{k}
\newcommand{\LGd}{\mathfrak{g}^*}
\newcommand{\rhocheck}{\rho^{\vee}}
\newcommand{\unipotentRadicalOfPARABOLICSUBGROUP}{U_P}
\newcommand{\lieAlgebraOfUnipotentRadicalOfPARABOLICSUBGROUP}{\mathfrak{u}_P}
\newcommand{\affinizationOfGrothendieckSpringerResolution}{\LGd \times_{\LTd\sslash W} \LTd}
\newcommand{\smoothLocusOfAffineclosureBasicAffineSpace}{\mathcal{S}}

\newcommand\blfootnote[1]{%
  \begingroup
  \renewcommand\thefootnote{}\footnote{#1}%
  \addtocounter{footnote}{-1}%
  \endgroup
}

\maketitle
\section{Introduction}
The goal of this paper is to prove the following theorem, which in particular confirms a conjecture of Ginzburg and Kazhdan \cite[Conjecture 1.3.6]{GinzburgKazhdanDifferentialOperatorsOnBasicAffineSpaceandtheGelfandGraevAction}:

\begin{Theorem}\label{Affine Closure of Cotangent Bundle of Basic Affine Space Has Symplectic Singularities}
    The variety $\affineClosureOfCotangentBundleofBasicAffineSpace$ has conical symplectic singularities.
\end{Theorem}
Here, $G$ denotes a connected semisimple group over $\mathbb{C}$, $N = [B, B]$ denotes the unipotent radical of some Borel $B$, and $\affineClosureOfCotangentBundleofBasicAffineSpace$ is the affinization of the quasi-affine variety $T^*(G/N)$. We also prove in \cref{affineClosureOfCotangentBundleofBasicAffineSpace has symplectic singularities for ARBITRARY REDUCTIVE GROUP} that $\affineClosureOfCotangentBundleofBasicAffineSpace$ has symplectic singularities if $G$ is reductive, although the $\mathbb{G}_m$-action we construct is not conical if $G$ is not semisimple. 
\vspace{0.02in}

\cref{Affine Closure of Cotangent Bundle of Basic Affine Space Has Symplectic Singularities} implies that $\affineClosureOfCotangentBundleofBasicAffineSpace$ conjecturally admits a \textit{symplectic dual} in the sense of Braden-Licata-Proudfoot-Webster. We also prove the following theorem, which determines properties of this conjectural dual:

\begin{Theorem}\label{affineClosureOfCotangentBundleofBasicAffineSpace is Q-Factorial Terminal and is Factorial When G is Simply Connected}
    The variety $\affineClosureOfCotangentBundleofBasicAffineSpace$ is $\mathbb{Q}$-factorial and has terminal singularities. Moreover, if $G$ is simply connected then the ring of functions of $\affineClosureOfCotangentBundleofBasicAffineSpace$ is a unique factorization domain. 
\end{Theorem}

We refer to \cite{JiaTheGeometryOfTheAffineClosureOfCotangentBundleOfBasicAffineSpaceForSLn}, \cite{WangANewWeylGroupActionRelatedTotheQuasiClassicalGelfandGraevAction}, and \cite{GinzburgKazhdanDifferentialOperatorsOnBasicAffineSpaceandtheGelfandGraevAction} for some motivation for \cref{Affine Closure of Cotangent Bundle of Basic Affine Space Has Symplectic Singularities} and for studying $\affineClosureOfCotangentBundleofBasicAffineSpace$ more generally, and to \cite{DancerHanayKirwanSymplecticDualityandImplosions}, \cite{BradenProudfootWebsterQuantizationsofConicalSymplecticResolutionsILocalandGlobalStructure}, \cite{KamnitzerSymplecticResolutionsSymplecticDualityandCoulombBranches}, \cite{BourgetDancerGrimmingerHanayZhongPartialImplosionsAndQuivers}, \cite{BradenLicataProudfootWebsterQuantizationsofConicalSymplecticResolutionsIICategoryOandSymplecticDuality}, \cite{BourgetDancerGrimmingerHanayKirwanZhongOrthosymplecticImplosions} for some motivation for symplectic duality and the relevance of $\affineClosureOfCotangentBundleofBasicAffineSpace$ to the mirror symmetry program. 

    \subsection{Outline of Proof} In  \cite{JiaTheGeometryOfTheAffineClosureOfCotangentBundleOfBasicAffineSpaceForSLn}, it is shown that $\affineClosureOfCotangentBundleofBasicAffineSpace$ has symplectic singularities when $G = \SL_n$. We briefly review the approach of \cite{JiaTheGeometryOfTheAffineClosureOfCotangentBundleOfBasicAffineSpaceForSLn} so as to compare and contrast with the approach taken in the general case here. When $G = \SL_n$, the variety $\affineClosureOfCotangentBundleofBasicAffineSpace$ admits a description as a hyperk\"ahler reduction of a certain vector space obtained from a quiver \cite{DancerKirawanSwannImplosionForHyperKahlerManifolds}. This quiver description gives a stratification of $\affineClosureOfCotangentBundleofBasicAffineSpace$ by hyperk\"ahler varieties of even (complex) codimension, and this stratification in particular includes a dense open smooth subset denoted in \cite{DancerKirawanSwannImplosionForHyperKahlerManifolds} as $Q^{\text{hks}}$. A key insight in \cite{JiaTheGeometryOfTheAffineClosureOfCotangentBundleOfBasicAffineSpaceForSLn} is that complement of $Q^{\text{hks}}$ in  $\affineClosureOfCotangentBundleofBasicAffineSpace$ in fact has codimension \textit{four}, which is proved by showing all other strata have codimension at least four. Results of \cite{StratenSteenbrinkExtendabilityofHolomorphicDifferentialFormsNearIsolatedHypersurfaceSingularities},  \cite{FlennerExtendabilityofDifferentialFormsonNonisolatedSingularities} then show that, to prove that $\affineClosureOfCotangentBundleofBasicAffineSpace$ has symplectic singularities, it suffices to show that $\affineClosureOfCotangentBundleofBasicAffineSpace$ is normal and that the smooth locus of $\affineClosureOfCotangentBundleofBasicAffineSpace$ admits a symplectic form, see \cref{Normal Variety with Nondeg Form on Enough of Regular Locus with High Dimensional Singular Locus Has Symplectic Singularities}.{\let\thefootnote\relax\footnotetext{\textit{Mathematics Subject Classification 2020:} 14B05, 22E57, 22E10}} These facts are proved in \cite{JiaTheGeometryOfTheAffineClosureOfCotangentBundleOfBasicAffineSpaceForSLn} using the above stratification and the Marsden-Weinstein theorem. 

   A quiver type description of $\affineClosureOfCotangentBundleofBasicAffineSpace$ is not expected to exist for general $G$, even when $G = \text{SO}_{2n}$ \cite{DancerHanayKirwanSymplecticDualityandImplosions}. In proving \cref{Affine Closure of Cotangent Bundle of Basic Affine Space Has Symplectic Singularities}, we circumvent this issue by constructing a smooth open subscheme $Q \subseteq \affineClosureOfCotangentBundleofBasicAffineSpace$ that we expect can be identified with $Q^{\text{hks}}$ when $G = \SL_n$. We show that the complement of $Q$ in $\affineClosureOfCotangentBundleofBasicAffineSpace$ has codimension at least four (\cref{Complement of Locus of Points on Which T Acts Freely for ARBITRARY REDUCTIVE GROUP Is at least Codimension Four}), and argue directly that $\affineClosureOfCotangentBundleofBasicAffineSpace$ is normal and that the regular locus of $\affineClosureOfCotangentBundleofBasicAffineSpace$ admits a symplectic form. Using this, we show that $\affineClosureOfCotangentBundleofBasicAffineSpace$ has symplectic singularities in \cref{Symplectic Singularities of Normal Varieties With Small Singular Locus Subsection}. 
   
To show that the complement of $Q$ has codimension at least four, we analyze the right $T$-action on $\affineClosureOfCotangentBundleofBasicAffineSpace$ and use the fact that the ring of functions on $\affineClosureOfCotangentBundleofBasicAffineSpace$ is a unique factorization domain when $G$ is simply connected, see \cref{The Ring ringOfFunctionsForCOTANGENTBUNDLEOfBasicAffineSpace is a UFD}. In this case, the fact that the singular locus of $\affineClosureofBasicAffineSpace$ has codimension at least four allows us to show that a set strictly smaller than $Q$ (in general) has codimension at least four, see \cref{Locus of Points of affineClosureOfCotangentBundleofBasicAffineSpace whose T-stabilizer has dimension geq 2 has codimension at least four}. This of course gives the fact that the complement of $Q$ has codimension at least four for simply connected $G$, which we use to derive the corresponding fact for general $G$ in \cref{Free Locus from Simply Connected Case Subsection}. The fact that $\affineClosureOfCotangentBundleofBasicAffineSpace$ has terminal singularities then follows immediately from a result of Namikawa \cite{NamikawaANoteOnSymplecticSingularities}.

    \subsection{Acknowledgments} I would like to thank Desmond Coles, Gurbir Dhillon, Louis Esser, Joakim F\ae rgeman, Victor Ginzburg, Rok Gregoric, Boming Jia, Frances Kirwan, Joaqu\'in Moraga, Morgan Opie, Alberto San Miguel Malaney, Kendric Schefers, and Jackson Van Dyke for many interesting and useful conversations. I would also like to thank Boming Jia and Sam Raskin for pointing out a gap in a draft version of this paper, as well as the anonymous referee, whose extremely detailed feedback greatly improved the exposition of this paper. Finally, I would also like to thank Ivan Losev and the organizers of the \textit{Quantized symplectic singularities and applications to Lie theory} 2022 conference, who provided an excellent environment for me to learn the basics of the theory of symplectic singularities.
\section{Recollections on The Affine Closure of The Basic Affine Space}\label{Recollections on The Affine Closure of The Basic Affine Space Section} We now collect the facts we will use regarding $\affineClosureofBasicAffineSpace$ and $\affineClosureOfCotangentBundleofBasicAffineSpace$ and set the notation which will be used in what follows. None of the material in \cref{Recollections on The Affine Closure of The Basic Affine Space Section} is original. More details, references, and proofs can be found in works such as \cite{BezrukavnikovBravermanPositselskiiGluingofAbelianCategoriesandDifferentialOperatorsontheBasicAffineSpace}, \cite{GinzburgRicheDifferentialOperatorsOnBasicAffineSpaceandtheAffineGrassmannian}, \cite{BraKazSchwartz}, \cite{GinzburgKazhdanDifferentialOperatorsOnBasicAffineSpaceandtheGelfandGraevAction}, and \cite{LevasseurStaffordDifferentialOperatorsandCohomologyGroupsontheBasicAffineSpace}. 

\subsection{Affine Closures of Basic Affine Space and Its Cotangent Bundle} \label{Affine Closures of Basic Affine Space and Its Cotangent Bundle Subsection}Hereafter, in every section except\footnote{In \cref{Proof of Main Theorem for Simply Connected G Section}, we will restrict to the case where $G_{\mathbb{Z}}$ is simply connected.} \cref{Proof of Main Theorem for Simply Connected G Section}, we let $G_{\mathbb{Z}}$ denote some split reductive group over $\mathbb{Z}$ with choice of maximal torus $T_{\mathbb{Z}}$, and let $G := G_{k}$ and $T := T_k$ denote the respective base changes to $k := \mathbb{C}$. Denote by $\characterlatticeforT$ denote the lattice of characters for $T$, and let $X_{\bullet}(T)$ denote the lattice of cocharacters. Choose some Borel subgroup $B \supseteq T$, and let $N$ denote the unipotent radical of $B$. Let $\LG$ and $\LT$ denote the Lie algebras of $G$ and $T$ respectively, and let $\LGd$ and $\LTd$ denote their respective dual Lie algebras. We will occasionally abuse notation by denoting $\LT(\mathbb{Q}) := \text{Lie}(T_{\mathbb{Q}})(\mathbb{Q})$ and $\LTd(\mathbb{Q}) := \text{Lie}(T_{\mathbb{Q}})^*(\mathbb{Q})$. With this notation, we have isomorphisms \begin{equation}\label{characterlatticeforT tensors to give Lie Algebra}
            \characterlatticeforT \otimes_{\mathbb{Z}} \mathbb{Q} \xrightarrow{\sim} \LTd(\mathbb{Q})
        \end{equation} and \begin{equation}\label{cocharacterlatticeforT tensors to give Lie Algebra}
            X_{\bullet}(T) \otimes_{\mathbb{Z}} \mathbb{Q} \xrightarrow{\sim} \LT(\mathbb{Q})
        \end{equation} both induced by the differential. 

By the algebraic Peter-Weyl theorem the ring of functions on $G/N$ \[\ringOfFunctionsForBasicAffineSpace := \Gamma(G/N, \mathcal{O}_{G/N})\] admits a grading by $\characterlatticeforT$ such that $\mathbb{C} \xrightarrow{\sim} A_0$ and $A_{\lambda} \neq 0$ only if $\lambda$ is dominant. Moreover, it is standard (see, for example, \cite[Proposition 14.26]{MilneAlgebraicGroupsBook}) that the variety $G/N$ is quasi-affine; we also reprove this fact below as an instance of a more general claim in \cref{G Mod Commutator of Parabolic Is Qaff With Stratification}. Therefore $G/N$ is an open subscheme of its \textit{affine closure} $\affineClosureofBasicAffineSpace := \Spec(A)$. 
    
    The projection map $\pi: T^*(G/N) \to G/N$ is affine, and so the variety $T^*(G/N)$ is also quasi-affine. In particular, we may identify it as an open subscheme of its affinization $\affineClosureOfCotangentBundleofBasicAffineSpace = \Spec(\ringOfFunctionsForCOTANGENTBUNDLEOfBasicAffineSpace)$, where \[\ringOfFunctionsForCOTANGENTBUNDLEOfBasicAffineSpace := \Gamma(T^*(G/N), \mathcal{O}_{T^*(G/N)})\] denotes the ring of global functions on $T^*(G/N)$. The map $\pi$ induces a canonical map \[\projectionFromAffineClosureofCotangentBundleToAffineClosureofSpace: \affineClosureOfCotangentBundleofBasicAffineSpace \to \affineClosureofBasicAffineSpace\] induced by the inclusion $\ringOfFunctionsForBasicAffineSpace \subseteq \ringOfFunctionsForCOTANGENTBUNDLEOfBasicAffineSpace$. The action of $G \times T$ on $G/N$ gives a corresponding $G \times T$ action on $\affineClosureOfCotangentBundleofBasicAffineSpace$. We refer to the action of the subgroup $1 \times T$ as \lq the\rq{} $T$-action on $\affineClosureOfCotangentBundleofBasicAffineSpace$ and, for any $\lambda \in \characterlatticeforT$, we let $\ringOfFunctionsForCOTANGENTBUNDLEOfBasicAffineSpace_{\lambda}$ denote the graded summand of $\ringOfFunctionsForCOTANGENTBUNDLEOfBasicAffineSpace$ induced from this $T$-action. 

\subsection{Moment Map Notation}
From the $G \times T$-action on $G/N$, we obtain a moment map $T^*(G/N) \to \LG^* \times \LTd$ which lifts to a map $T^*(G/N) \to \LG^* \times_{\LTd\sslash W} \LTd$. Since $\LG^* \times_{\LTd\sslash W} \LTd$ is affine, we obtain an induced map \[\momentMapFromAFFINECLOSUREofCotangentSpacewithTOTALGROUP: \affineClosureOfCotangentBundleofBasicAffineSpace \to \LG^* \times_{\LTd\sslash W} \LTd\] by the universal property of affinization. 

\subsection{Algebraic Gelfand-Graev Action on $\affineClosureOfCotangentBundleofBasicAffineSpace$}
    We recall a weaker form of one of the main theorems of \cite{GinzburgKazhdanDifferentialOperatorsOnBasicAffineSpaceandtheGelfandGraevAction}, see also \cite[Section 5.5]{GinzburgRicheDifferentialOperatorsOnBasicAffineSpaceandtheAffineGrassmannian}:

\begin{Theorem}\label{Gelfand-Graev Action on Affine Closure of Cotangent Bundle} \cite[Section 1.3]{GinzburgKazhdanDifferentialOperatorsOnBasicAffineSpaceandtheGelfandGraevAction}
    There is a canonical $G, T \rtimes W$ action on $\affineClosureOfCotangentBundleofBasicAffineSpace$ lifting the $G \times T$ action such that the map $\momentMapFromAFFINECLOSUREofCotangentSpacewithTOTALGROUP$ is $W$-equivariant.
\end{Theorem}

\subsection{Finite Generation of Functions on $T^*(G/N)$} We now summarize some results of \cite[Section 3.6]{GinzburgRicheDifferentialOperatorsOnBasicAffineSpaceandtheAffineGrassmannian} which will be used below. For any $w \in W$, let $\projectionFromAffineClosureofCotangentBundleToAffineClosureofSpace_w$ denote the composite map $\projectionFromAffineClosureofCotangentBundleToAffineClosureofSpace \circ w: \affineClosureOfCotangentBundleofBasicAffineSpace \to \affineClosureofBasicAffineSpace$. 

\begin{Lemma}\label{affineClosureOfCotangentBundleofBasicAffineSpace is a Closed Subscheme of |W| Many Products of G/N times LGd timesLTdsslash W LTD and Multiplication Map Is Surjective on Graded Pieces}
The ring $\ringOfFunctionsForCOTANGENTBUNDLEOfBasicAffineSpace$ is finitely generated, and in particular Noetherian. Moreover, the morphism \[\affineClosureOfCotangentBundleofBasicAffineSpace \xrightarrow{\momentMapFromAFFINECLOSUREofCotangentSpacewithTOTALGROUP \times \bigtimes_w \projectionFromAffineClosureofCotangentBundleToAffineClosureofSpace_w}\LGd \times_{\LTd\sslash W} \LTd \times \bigtimes_{w \in W} \affineClosureofBasicAffineSpace\] is a $G \times (T \rtimes W)$-equivariant closed embedding, and, for any dominant $\lambda \in \characterlatticeforT$ and any $w \in W$, the restricted multiplication map \[M_{w, \lambda}: \Sym(\LG) \otimes_{Z\LG} \Sym(\LT) \otimes_{\groundfield} w(\ringOfFunctionsForBasicAffineSpace_{\lambda}) \to \ringOfFunctionsForCOTANGENTBUNDLEOfBasicAffineSpace_{w\lambda}\] is surjective, where $w(\ringOfFunctionsForBasicAffineSpace_{\lambda})$ denotes the image of $\ringOfFunctionsForBasicAffineSpace_{\lambda}$ in $\ringOfFunctionsForCOTANGENTBUNDLEOfBasicAffineSpace$ under $w$. 
\end{Lemma}

\section{Ring of Functions on $T^*(G/N)$} 
In this section, we study the variety $\affineClosureOfCotangentBundleofBasicAffineSpace$ and its ring of functions $\ringOfFunctionsForCOTANGENTBUNDLEOfBasicAffineSpace$. We first show that the variety $\affineClosureOfCotangentBundleofBasicAffineSpace$ is $\mathbb{Q}$-factorial, and moreover that $\ringOfFunctionsForCOTANGENTBUNDLEOfBasicAffineSpace$ is a UFD when $G$ is simply connected, in \cref{Factoriality and Q-Factoriality of affineClosureOfCotangentBundleofBasicAffineSpace Subsection}. We then construct a $\mathbb{G}_m$-action on $\affineClosureOfCotangentBundleofBasicAffineSpace$ in \cref{Grading on Functions on T*(G/N) Subsection} and record some of its basic properties. 

\subsection{Factoriality and $\mathbb{Q}$-Factoriality of Affine Closure of $T^*(G/N)$}\label{Factoriality and Q-Factoriality of affineClosureOfCotangentBundleofBasicAffineSpace Subsection}
Recall that a normal variety is said to be $\mathbb{Q}$\textit{-factorial} if the cokernel of the map $\text{Pic}(X) \xhookrightarrow{} \text{Cl}(X)$ is torsion. We now show:

\begin{Proposition}\label{affineClosureOfCotangentBundleofBasicAffineSpace is normal and Q-factorial for all G and is factorial when G is simply connected}
    The variety $\affineClosureOfCotangentBundleofBasicAffineSpace$ is $\mathbb{Q}$-factorial and, in particular, normal. 
\end{Proposition}

Notice that $\affineClosureofBasicAffineSpace$ and $\affineClosureOfCotangentBundleofBasicAffineSpace$ are normal since the ring of functions on any smooth variety is normal, see, for example, \cite[Lemma 28.7.9]{StacksProject}. To prove the remainder of \cref{affineClosureOfCotangentBundleofBasicAffineSpace is normal and Q-factorial for all G and is factorial when G is simply connected}, we first show the following: 

\begin{Lemma}\label{Vanishing of Class and Picard Group for Affine Closure of G/N}
    The class group and the Picard group of $\overline{G/N}$ are finite. Moreover, if $G$ is simply connected, both groups are trivial. 
\end{Lemma}

\begin{proof}
    For any normal variety, the Picard group injects into the class group, so it suffices to show the class group of $\overline{G/N}$ is finite and, when $G$ is simply connected, is trivial. The class group of $\overline{G/N}$ agrees with the class group of $G/N$ since its complement has codimension 2 by the stratification \labelcref{Stratification of Affine Closure of Basic Affine Space by Partial Basic Affine Spaces}, see, for example \cite[Proposition II.6.5(b)]{HartshorneAlgebraicGeometry}. Since $G/N$ is smooth, its class group and Picard group agree. However, it is known (see, for example, \cite[Theorem 18.32]{MilneAlgebraicGroupsBook}) that we have an exact sequence \[X^{\bullet}(N) \to \text{Pic}(G/N) \to \text{Pic}(G) \to \text{Pic}(N)\] 

\noindent where $X^{\bullet}(N) := \text{Hom}_{\text{AlgGp}}(N, \mathbb{G}_m)$. Notice that $X^{\bullet}(N) = 0$ and $\text{Pic}(N) = 0$ as $N$ is unipotent by \cite[Corollary 14.18]{MilneAlgebraicGroupsBook} and \cite[Proposition 14.32]{MilneAlgebraicGroupsBook}, respectively. Therefore $\text{Pic}(G/N) \xrightarrow{\sim} \text{Pic}(G)$. Now our claim follows from the fact that $\text{Pic}(G)$ is finite \cite[Corollary 18.23]{MilneAlgebraicGroupsBook} and the fact that if $G$ is simply connected then $\text{Pic}(G)$ is trivial \cite[Corollary 18.24]{MilneAlgebraicGroupsBook}. 
\end{proof}

From \cref{Vanishing of Class and Picard Group for Affine Closure of G/N}, we derive the following result, which in particular completes the proof of \cref{affineClosureOfCotangentBundleofBasicAffineSpace is normal and Q-factorial for all G and is factorial when G is simply connected}:
\begin{Proposition}\label{Class Group of AffineClosureofCotangentBundleofBasicAffineSpace Vanishes}
    The class group of $\affineClosureOfCotangentBundleofBasicAffineSpace$ is isomorphic to the class group of $\affineClosureofBasicAffineSpace$. In particular, the class group of $\affineClosureOfCotangentBundleofBasicAffineSpace$ is finite and, if $G$ is simply connected, is trivial.
\end{Proposition}
We show this after showing the following lemma: 
\newcommand{\arbitraryQuasiAffineScheme}{Y}
\newcommand{\ringOfFunctionsOnArbitraryQuasiAffineScheme}{B}
\newcommand{\specOfRingOfFunctionsOnArbitraryQuasiAffineScheme}{\Spec(\ringOfFunctionsOnArbitraryQuasiAffineScheme)}
\begin{Lemma}\label{Complement to QuasiAffine Noetherian Scheme Must Be of Codimension At Least Two}
Assume $\arbitraryQuasiAffineScheme$ is an integral quasi-affine scheme and let $\ringOfFunctionsOnArbitraryQuasiAffineScheme$ denote its ring of functions so that $j: \arbitraryQuasiAffineScheme \to \specOfRingOfFunctionsOnArbitraryQuasiAffineScheme$ is an open embedding. If $\ringOfFunctionsOnArbitraryQuasiAffineScheme$ is Noetherian, then the complement of $\arbitraryQuasiAffineScheme$ in $\specOfRingOfFunctionsOnArbitraryQuasiAffineScheme$ has codimension at least two. 
\end{Lemma}

\begin{proof}
    Choose some minimal prime $\mathfrak{p}$ in the complement of $\arbitraryQuasiAffineScheme$, and let $Z$ denote the integral scheme $\Spec(\ringOfFunctionsOnArbitraryQuasiAffineScheme/\mathfrak{p})$. Letting $U := \specOfRingOfFunctionsOnArbitraryQuasiAffineScheme\setminus Z$ and $X := \specOfRingOfFunctionsOnArbitraryQuasiAffineScheme$, we have a containment of open subschemes \begin{equation}
        \label{Containment of Various Quasiaffine Schemes}
    X \supseteq U \supseteq \arbitraryQuasiAffineScheme \end{equation} which induces a map $j^{\#}: H^0(U; \mathcal{O}_U) \xrightarrow{} H^0(\arbitraryQuasiAffineScheme; \mathcal{O}_\arbitraryQuasiAffineScheme) = \ringOfFunctionsOnArbitraryQuasiAffineScheme$. This map is surjective by the induced map of functions given by \labelcref{Containment of Various Quasiaffine Schemes} and injective since $\arbitraryQuasiAffineScheme$, and thus $\specOfRingOfFunctionsOnArbitraryQuasiAffineScheme$, are integral, and thus we see that $j^{\#}$ is an isomorphism. From this and \labelcref{Containment of Various Quasiaffine Schemes}, it follows that the restriction map \[\mathcal{O}_X(X) \xrightarrow{} \mathcal{O}_X(U)\] is an isomorphism. Note that $U$ is quasi-compact, as it is an open subset of $\specOfRingOfFunctionsOnArbitraryQuasiAffineScheme$ for $\ringOfFunctionsOnArbitraryQuasiAffineScheme$ Noetherian. However, for quasi-compact $U$, the main result of \cite{MoretBaillyComplementofADivisorInAnAffineScheme} gives that \labelcref{Containment of Various Quasiaffine Schemes} is not an isomorphism if $\mathfrak{p}$ is a divisor, so the height of $\mathfrak{p}$ must be at least two. 
\end{proof}

\begin{proof}[Proof of \cref{Class Group of AffineClosureofCotangentBundleofBasicAffineSpace Vanishes}]
    Since the complement of $T^*(G/N)$ has codimension at least two in $\affineClosureOfCotangentBundleofBasicAffineSpace$ by \cref{Complement to QuasiAffine Noetherian Scheme Must Be of Codimension At Least Two}, it suffices to show the analogous claim for $T^*(G/N)$. Since $T^*(G/N)$ is smooth, by the Auslander-Buchsbaum theorem its Picard group and class group agree. Therefore by \cref{Vanishing of Class and Picard Group for Affine Closure of G/N} it is enough to show that the map \[\pi^*: \text{Pic}(G/N) \to \text{Pic}(T^*(G/N))\] of abelian groups is an isomorphism. This map is injective, as any line bundle $\mathcal{L}$ for which $\pi^*(\mathcal{L})$ is trivial has the property that $\mathcal{L} \cong z^*(\pi^*(\mathcal{L}))$ is also trivial, where $z$ denotes the zero section. The surjectivity follows from \cite[Corollaire IV.21.4.11, Erratum]{DieudonneGrothendieckEGA}, as $\pi: T^*(G/N) \to G/N$ is a faithfully flat morphism to a normal variety whose its fibers are vector spaces. %
\end{proof}

Since the class group of $\Spec(\ringOfFunctionsForCOTANGENTBUNDLEOfBasicAffineSpace) = \affineClosureOfCotangentBundleofBasicAffineSpace$ is trivial when $G$ is simply connected, we immediately obtain: 

\begin{Corollary}\label{The Ring ringOfFunctionsForCOTANGENTBUNDLEOfBasicAffineSpace is a UFD}
    The ring $\ringOfFunctionsForCOTANGENTBUNDLEOfBasicAffineSpace$ is a unique factorization domain if $G$ is simply connected. 
\end{Corollary}

\subsection{Grading on Functions on $T^*(G/N)$}\label{Grading on Functions on T*(G/N) Subsection}
We can define a $\mathbb{G}_m$-action on $\affineClosureOfCotangentBundleofBasicAffineSpace$ defined as follows. Let 2$\rho^{\vee}$ denote the product of the positive coroots in $T$. This defines a map $p: \mathbb{G}_m \to T$ which we denote $\alpha \mapsto h_{\alpha}$. Define a $\mathbb{G}_m$-action on $G \times \LB$ via \[(g, \xi)\alpha := (gh_{\alpha}, \alpha^2\text{Ad}_{h_{\alpha}^{-1}}(\xi)).\] Since, for any $u \in N$ \[(guh_{\alpha}, \alpha^2\text{Ad}_{h_{\alpha}^{-1}}\text{Ad}_{u^{-1}}(\xi)) = (gh_{\alpha}(h_{\alpha}^{-1}uh_{\alpha}), \alpha^2\text{Ad}_{h_{\alpha}^{-1}u^{-1}h_{\alpha}h_{\alpha}^{-1}}(\xi)) = (gh_{\alpha}u_0, \alpha^2\text{Ad}_{u_0^{-1}}\text{Ad}_{h_{\alpha}^{-1}}(\xi))\] where $u_0 := h_{\alpha}^{-1}uh_{\alpha}$, we see that this gives an action of $\mathbb{G}_m$ on $T^*(G/N) \cong G \mathop{\times}\limits^{N} \LB$ and in particular equips $\ringOfFunctionsForCOTANGENTBUNDLEOfBasicAffineSpace$ with a grading $\ringOfFunctionsForCOTANGENTBUNDLEOfBasicAffineSpace = \oplus_i\ringOfFunctionsForCOTANGENTBUNDLEOfBasicAffineSpace^i$, where we use superscripts for the grading to disambiguate from the $\characterlatticeforT$-grading on $\ringOfFunctionsForCOTANGENTBUNDLEOfBasicAffineSpace$ of \cref{Affine Closures of Basic Affine Space and Its Cotangent Bundle Subsection}. For a nonzero homogeneous element $r \in \ringOfFunctionsForCOTANGENTBUNDLEOfBasicAffineSpace$, with respect to this grading, we let $c(r)$ denote the unique integer with $r \in \ringOfFunctionsForCOTANGENTBUNDLEOfBasicAffineSpace^{c(r)}$, and refer to $c(r)$ as the $c$\textit{-grading} of $r$. We use this term since, when $G$ is semisimple, this grading is conical, as stated in the last point of the following proposition: 

\begin{Proposition}\label{Grading on Ring of Functions is (1) Admits Equivariant Projection and Moment Maps (2) Is Determined by T Weight and Height (3) is W-equivariant (4) is conical}
    The above grading on $\ringOfFunctionsForCOTANGENTBUNDLEOfBasicAffineSpace$ has the following properties:
    \begin{enumerate}
        \item The maps $\projectionFromAffineClosureofCotangentBundleToAffineClosureofSpace$ and $\momentMapFromAFFINECLOSUREofCotangentSpacewithTOTALGROUP$ are $\mathbb{G}_m$-equivariant, where $\affineClosureofBasicAffineSpace$ is equipped with a $\mathbb{G}_m$-action via restricting the $T$-action via $p$ and $\mathbb{G}_m$ acts on $\LGd \times_{\LTd\sslash W} \LTd$ via $\alpha(\xi, \nu) = (\alpha^2\xi, \alpha^2\nu)$. 
        \item For any nonzero $r \in \ringOfFunctionsForCOTANGENTBUNDLEOfBasicAffineSpace_{\lambda}$ of usual grading $h_r$, $r$ is homogeneous with respect to the $c$-grading and moreover $c(r) = 2h_r + \langle \lambda, 2\rhocheck \rangle$.
        \item The grading is $W$-equivariant--that is, $c(r) = c(wr)$ for any homogeneous nonzero $r \in \ringOfFunctionsForCOTANGENTBUNDLEOfBasicAffineSpace$ and $w \in W$. 
        \item If $G$ is semisimple, then the $c$-grading on $\ringOfFunctionsForCOTANGENTBUNDLEOfBasicAffineSpace$ is \textit{conical}--that is, $R^i$ is nonzero only if $i \geq 0$ and $R^0 = k$. 
    \end{enumerate}
\end{Proposition}

\begin{proof}
    The first claim immediately follows from the fact that $\pi$ and $\mu$ are $\mathbb{G}_m$-equivariant. By \cref{affineClosureOfCotangentBundleofBasicAffineSpace is a Closed Subscheme of |W| Many Products of G/N times LGd timesLTdsslash W LTD and Multiplication Map Is Surjective on Graded Pieces} and (1), we see that (2) holds for any $\lambda$ lying in the closure of the Weyl chamber $\overline{C} = 1\overline{C}$ determined by our choices of $B$ and $N$. Using this as the base case, we now proceed by induction on the length of $w \in W$. For any $\lambda \in w\overline{C}$, we may choose some simple reflection $s$ such that $sw < w$. Let $r \in \ringOfFunctionsForCOTANGENTBUNDLEOfBasicAffineSpace_{\lambda}$. By induction we see that $s(r)$ and $rs(r)$ are both homogeneous with respect to the $c$-grading, and thus $r$ is as well since $\ringOfFunctionsForCOTANGENTBUNDLEOfBasicAffineSpace$ is an integral domain. Letting $h_{s(r)}$ and $h_{r}$ denote the respective usual gradings we then obtain \[c(r) + 2h_{s(r)} + \langle s(\lambda), 2\rhocheck \rangle = c(r) + c(s(r)) = c(rs(r)) = 2(h_{s(r)} + h_{r}) + \langle \lambda + s(\lambda), 2\rhocheck\rangle\] where both the first and last step use the inductive hypothesis. We therefore see \[c(r) = 2h_r + \langle \lambda, 2\rhocheck\rangle\] which shows (2). 

    To prove $(3)$, it suffices to show the claim on some generating set. By \cref{affineClosureOfCotangentBundleofBasicAffineSpace is a Closed Subscheme of |W| Many Products of G/N times LGd timesLTdsslash W LTD and Multiplication Map Is Surjective on Graded Pieces}, we may choose this generating set whose elements are a basis of $\LG \oplus \LT$ as well as the elements $w(a)$ for all $a \in A_{\lambda}$ for $\lambda$ dominant and $w \in W$. The former case follows from (1) and the $W$-equivariance of $\momentMapFromAFFINECLOSUREofCotangentSpacewithTOTALGROUP$, so we may assume $r = w(a)$ for some $a \in A_{\lambda}$. However, for such $r$, it is known \cite[Remark 3.2.2(1)]{GinzburgRicheDifferentialOperatorsOnBasicAffineSpaceandtheAffineGrassmannian}, building on the analogous claim for differential operators \cite[Proposition 2.9]{LevasseurStaffordDifferentialOperatorsandCohomologyGroupsontheBasicAffineSpace}, \cite{BezrukavnikovBravermanPositselskiiGluingofAbelianCategoriesandDifferentialOperatorsontheBasicAffineSpace}, that the usual grading of $w(a)$ is $\langle \lambda - w(\lambda), \rhocheck\rangle$. In particular, we see \[c(w(a)) = 2\langle \lambda - w(\lambda), \rhocheck\rangle + \langle w(\lambda), 2\rhocheck \rangle = \langle \lambda, 2\rhocheck \rangle\] is independent of $w$. Furthermore, when $G$ is semisimple, we have that $\langle \lambda, 2\rho^{\vee} \rangle > 0$ for all nonzero dominant $\lambda$. Therefore for such $G$ each element in the set \[\{w(\ringOfFunctionsForBasicAffineSpace_{\lambda}): w \in W\} \cup \LG \oplus \LT\] has positive $c$-grading, where $\lambda$ varies over the dominant nonzero weights. Since this set generates $\ringOfFunctionsForCOTANGENTBUNDLEOfBasicAffineSpace$ by \cref{affineClosureOfCotangentBundleofBasicAffineSpace is a Closed Subscheme of |W| Many Products of G/N times LGd timesLTdsslash W LTD and Multiplication Map Is Surjective on Graded Pieces}, we obtain (4). 
\end{proof}

\begin{Corollary}\label{Poisson Bracket for affineClosureOfCotangentBundleofBasicAffineSpace Has Degree -2} 
    The Poisson bracket $\{\cdot, \cdot\}$ on the algebra $\ringOfFunctionsForCOTANGENTBUNDLEOfBasicAffineSpace$ lowers $c$-grading by 2--that is, if $x \in \ringOfFunctionsForCOTANGENTBUNDLEOfBasicAffineSpace^i$ and $y \in \ringOfFunctionsForCOTANGENTBUNDLEOfBasicAffineSpace^j$, then $\{x, y\} \in \ringOfFunctionsForCOTANGENTBUNDLEOfBasicAffineSpace^{i + j - 2}$.  
\end{Corollary}

\begin{proof}
Fixing $x, y$ as above, by \cref{Grading on Ring of Functions is (1) Admits Equivariant Projection and Moment Maps (2) Is Determined by T Weight and Height (3) is W-equivariant (4) is conical}(2) we may assume that there exists $\lambda, \lambda' \in \characterlatticeforT$ and $h, h' \in \mathbb{Z}^{\geq 0}$ such that $i = \langle \lambda, 2\rhocheck\rangle + 2h$, $j = \langle \lambda', 2\rhocheck\rangle + 2h'$, and, with respect to the $\mathbb{G}_m$-action given by fiber dilation, the grading on $x$, respectively $y$, is $h$, respectively $h'$. The Poisson bracket on $\ringOfFunctionsForCOTANGENTBUNDLEOfBasicAffineSpace$ preserves the $T$-grading and lowers the (usual) degree of a vector field by one (which can be checked locally on $T^*(G/N)$). Therefore we see that $\{x, y\} \in R_{\lambda + \lambda'}$ and its grading from the $\mathbb{G}_m$-action given by fiber dilation is $h + h' - 1$, and so its $c$-grading is \[2(h + h' - 1) + \langle \lambda + \lambda', 2\rhocheck \rangle  = i + j - 2\] using \cref{Grading on Ring of Functions is (1) Admits Equivariant Projection and Moment Maps (2) Is Determined by T Weight and Height (3) is W-equivariant (4) is conical}(2), as desired. 
\end{proof}

\begin{Remark}
    This $\mathbb{G}_m$-action is also considered when $G = \SL_n$ in \cite[Section 5]{JiaTheGeometryOfTheAffineClosureOfCotangentBundleOfBasicAffineSpaceForSLn}. In particular, it is shown that this grading is compatible with a natural $\mathbb{G}_m$-action given by identifying $\overline{T^*(\text{SL}_n/N)}$ with a hyperk\"ahler reduction of a certain vector space \cite{DancerKirawanSwannImplosionForHyperKahlerManifolds}--see \cite[Proposition 5.5]{JiaTheGeometryOfTheAffineClosureOfCotangentBundleOfBasicAffineSpaceForSLn} for the precise statement.
\end{Remark}

\section{Preliminary Results in Algebraic Geometry}
\subsection{GIT Quotients of Integral Varieties} We now record two properties of GIT quotients of integral affine varieties with $\mathbb{G}_m$-actions which will be used below. 

\begin{Lemma}\label{If You're a Graded Integral Finitely Generated Domain with a Nonzero Graded Element and You Take the GIT Quotient by Gm You Lose a Dimension}
    Assume $S$ is a graded integral finitely generated $k$-algebra such that there is a positively graded or negatively graded homogeneous element $x$. Then $\text{dim}(\text{Spec}(S)\sslash \mathbb{G}_m) \leq \text{dim}(\text{Spec}(S)) - 1$, where we equip $\text{Spec}(S)$ with the $\mathbb{G}_m$-action corresponding to the grading on $S$.
\end{Lemma}

\begin{proof}
    The fact that $S$ is an integral domain implies that the multiplication map $S_0 \otimes_k k[x] \to S$ is injective, since we may check if an element of $S$ is nonzero by checking if each projection onto each homogeneous summand of $S$ is nonzero. We thus see that the morphism \[\text{Spec}(S) \to \text{Spec}(S_0 \otimes_k k[x]) \cong \text{Spec}(S_0) \times \mathbb{A}^1 \cong \text{Spec}(S)\sslash \mathbb{G}_m \times \mathbb{A}^1\] is dominant. Therefore we see that \[\text{dim}(\text{Spec}(S)\sslash \mathbb{G}_m) + 1 = \text{dim}(\text{Spec}(S)\sslash \mathbb{G}_m) + \text{dim}(\mathbb{A}^1) = \text{dim}(\text{Spec}(S)\sslash \mathbb{G}_m \times \mathbb{A}^1) \leq \text{dim}(\text{Spec}(S))\] where the second equality uses the fact that $S$ is finitely generated and therefore in particular Noetherian. 
\end{proof}

Recall that, for any affine variety $\Spec(S)$ with an action of $\mathbb{G}_m$, this action is determined by a $\mathbb{Z}$-grading on the ring $S$. Moreover, the fixed point subscheme $\Spec(S)^{\mathbb{G}_m}$ is a closed subscheme cut out by the ideal $I$ generated by the $S_i$ for $i \neq 0$. We have an induced map \[\Spec(S/I) = \Spec(S)^{\mathbb{G}_m} \to \Spec(S)\sslash \mathbb{G}_m = \Spec(S_0)\] by composing the inclusion and the quotient map. Moreover, the induced map of rings $S_0 \to S/I$ is surjective since $f = \sum_i f_i \in S$ is congruent to $f_0$ in $S/I$. This proves the following observation:

\begin{Proposition}\label{For Scheme with Gm Action Fixed Points Are Closed Subscheme of GIT Quotient}
    For any affine variety $\Spec(S)$ with an action of $\mathbb{G}_m$, the induced map $\Spec(S)^{\mathbb{G}_m} \to \Spec(S)\sslash \mathbb{G}_m$ is a closed embedding. 
\end{Proposition}
\subsection{Fiber of Projection Over Smooth Locus}
Next, we compute the fiber of $\projectionFromAffineClosureofCotangentBundleToAffineClosureofSpace$ over its smooth locus $\smoothLocusOfAffineclosureBasicAffineSpace$.  Notice that the complement of $T^*(G/N)$ in $T^*(\smoothLocusOfAffineclosureBasicAffineSpace)$ has codimension $\geq 2$ by the stratification \labelcref{Stratification of Affine Closure of Basic Affine Space by Partial Basic Affine Spaces}. From this, we see that $\mathcal{O}(T^*(\smoothLocusOfAffineclosureBasicAffineSpace)) \xrightarrow{\sim} \mathcal{O}(T^*(G/N))$, and therefore we obtain a canonical map $\sigma: T^*(\smoothLocusOfAffineclosureBasicAffineSpace) \to \affineClosureOfCotangentBundleofBasicAffineSpace \times_{\affineClosureofBasicAffineSpace} \smoothLocusOfAffineclosureBasicAffineSpace$. 

\begin{Proposition}\label{Cotangent Bundle of Smooth Locus is Fiber of affineClosureOfCotangentBundleofBasicAffineSpace Over Smooth Locus of G/N-bar}
    The map $\sigma$ is an isomorphism. 
\end{Proposition}

\begin{proof}
    For any variety $Y$, we let $\mathcal{T}_Y$ denote its tangent sheaf. Notice that we have a canonical map \[\affineClosureOfCotangentBundleofBasicAffineSpace \to \SpecOfSymofSectionsOfTangentBundleOfBasicAffineSpace \] over $\affineClosureofBasicAffineSpace$ given by the universal property of affinization and the fact that $\SpecOfSymofSectionsOfTangentBundleOfBasicAffineSpace$ is affine. Therefore, we obtain a canonical map $\phi$ given by the composite \[\affineClosureOfCotangentBundleofBasicAffineSpace \times_{\affineClosureofBasicAffineSpace} \smoothLocusOfAffineclosureBasicAffineSpace \to \SpecOfSymofSectionsOfTangentBundleOfBasicAffineSpace \times_{\affineClosureofBasicAffineSpace} \smoothLocusOfAffineclosureBasicAffineSpace \cong T^*(\smoothLocusOfAffineclosureBasicAffineSpace)\] where this isomorphism is given by the fact that $\smoothLocusOfAffineclosureBasicAffineSpace$ is smooth, for which the composite \[T^*(\smoothLocusOfAffineclosureBasicAffineSpace) \xrightarrow{\sigma} \affineClosureOfCotangentBundleofBasicAffineSpace \times_{\affineClosureofBasicAffineSpace} \smoothLocusOfAffineclosureBasicAffineSpace \xrightarrow{\phi} T^*(\smoothLocusOfAffineclosureBasicAffineSpace)\] is the identity. Since this map is a section and $\phi$ is separated 
    we see that $\sigma$ is a closed embedding of equidimensional integral schemes, and therefore is an isomorphism. 
\end{proof}

\section{The Singular Locus for Simply Connected $G$}\label{Proof of Main Theorem for Simply Connected G Section}
In this subsection, we study the singular locus of $\affineClosureOfCotangentBundleofBasicAffineSpace$ in the case where $G$ is simply connected. For such $G$ we also introduce the subset $Q$ and study its basic properties. To avoid excessive repetition, \textit{in all of \cref{Proof of Main Theorem for Simply Connected G Section} we assume that $G$ is simply connected}. 

\subsection{Stratification of Affine Closure of $G/N$}\label{Stratification of Affineclosureofbasicaffinespace Subsection}
For any subset of simple coroots or subset of simple roots $\theta$, we let $P_{\theta}$ denote the standard parabolic subgroup containing $B$ and labeled by $\theta$. The goal of \cref{Stratification of Affineclosureofbasicaffinespace Subsection} is to prove the following well known proposition whose proof we are unable to find in the literature:

\begin{Proposition}\label{Stratification of Affine Closure of Basic Affine Space by Partial Basic Affine Spaces Proposition}
    The variety $G/N$ is quasi-affine and its affine closure admits a stratification\begin{equation}
    \label{Stratification of Affine Closure of Basic Affine Space by Partial Basic Affine Spaces}    \affineClosureofBasicAffineSpace =  \bigcup_{\theta}G/[P_{\theta}, P_{\theta}] 
    \end{equation} into the orbits of the action of $G \times T$, where $\theta$ varies over all subsets of simple roots. Moreover, the smooth locus of $\affineClosureofBasicAffineSpace$ contains all locally closed subschemes $G/[P_{\theta}, P_{\theta}]$ where $|\theta| = 1$.
\end{Proposition}

The first sentence of \cref{Stratification of Affine Closure of Basic Affine Space by Partial Basic Affine Spaces Proposition}, whose proof heavily uses ideas of \cite{KempfOnTheCollapsingofHomogeneousBundles}, also holds by similar methods in a more general context where $N = [B, B]$ is replaced by the commutator of an arbitrary parabolic subgroup, see \cref{G Mod Commutator of Parabolic Is Qaff With Stratification}. We use this generalization to prove the second sentence of \cref{Stratification of Affine Closure of Basic Affine Space by Partial Basic Affine Spaces Proposition}. 
\subsubsection{Commutator of Parabolic}\label{Commutator of Parabolic Subsubsection} Choose a subset $I$ of the set of simple coroots $\Delta^{\vee}$ and let $J$ denote its complement. For each simple coroot $c$, we let $\omega_c \in \characterlatticeforT$ denote the fundamental weight dual to $c$ and choose a nonzero highest weight vector $\vec{v}_c$ in the irreducible representation $V_c$ of highest weight $\omega_c$. (Each $\vec{v}_i$ is of course unique up to nonzero scalar multiple and our constructions will be independent of our choices of $\vec{v}_i$.)  If we let $L_j$ denote the line spanned by $\vec{v}_j$ for every $j \in J$, then $P_{\Delta^{\vee}\setminus \{j\}}$ is stabilizer of $L_j$, and moreover $P_{I} = \cap_{j \in J} P_{\Delta^{\vee}\setminus \{j\}}$, both of which can be seen from for example the fact that any subgroup of $G$ containing our choice of Borel subgroup is a standard parabolic subgroup \cite[Theorem 29.3(a)]{HumphreysLinearAlgebraicGroups} by computing the simple root vectors in the Lie algebra of the group. (In particular, $P_{\emptyset} = B$.) When $I$ contains a single element $i$ we will also let $P_{i} := P_{\{i\}}$ and, when there is no danger of confusion, we will also denote $P_i$ by $P_{\alpha}$ where $\alpha$ is the simple root associated to $i$. 

We also define the $G$-representation $V_J := \oplus_{j \in J} V_j$, and let $E_J := \text{span}_{V_J}\{\vec{v}_j : j \in J\}$ which is a subspace of $V_J$ which is closed under the action of $P_I$ and has dimension $|J|$. We naturally obtain a representation \begin{equation}\label{Representation of Parabolic on Product of Lines} \rho_I: P_I \to \text{Aut}(E_J) = \mathbb{G}_m^{|J|}\end{equation} of $P_I$ whose restriction along the map \begin{equation}\label{Definition of alphasubJvee}\alpha_J^{\vee} := \prod_{\alpha^{\vee} \in J}\alpha^{\vee}: \mathbb{G}_m^{|J|} \to P_I\end{equation} is an isomorphism. Therefore $\alpha_J^{\vee}$ is a closed embedding, and moreover if we denote the image of $\alpha_J^{\vee}$ by $\mathbb{G}_m^J$ we see that by for example \cite[Proposition 2.34]{MilneAlgebraicGroupsBook} we have a semidirect product decomposition \begin{equation}\label{Semidirect Product Decomposition for Parabolics}P_I \xleftarrow{\sim} Q_I \rtimes \mathbb{G}_m^J\end{equation} where $Q_I$ denotes the kernel of the map of \labelcref{Representation of Parabolic on Product of Lines}, or equivalently the intersection of stabilizers of the $\vec{v}_j$ for $j \in J$. Moreover, since $\mathbb{G}_m^{|J|}$ is abelian, $Q_I$ contains $[P_I, P_I]$ and, since $[P_I, P_I]$ and $Q_I$ are affine (as they are closed subgroup schemes of $G$), the inclusion map $[P_I, P_I] \xhookrightarrow{} Q_I$ is a closed embedding. 

\begin{Proposition}\label{Commutator of Parabolic is Stabilizer of Vector}
    The closed embedding $[P_I, P_I] \xhookrightarrow{} Q_I$ is an isomorphism.
\end{Proposition}

\begin{proof}
    By \labelcref{Semidirect Product Decomposition for Parabolics} we have an isomorphism of varieties $P_I/\mathbb{G}_m^J \xleftarrow{\sim} Q_I$, so $Q_I$ is connected. A theorem of Cartier gives any affine algebraic group in characteristic zero is smooth (see for example \cite[Theorem 3.23]{MilneAlgebraicGroupsBook}) and so any connected affine algebraic group over our characteristic zero field is in particular irreducible; it thus suffices to show that this closed embedding induces an equality of the Lie algebras of $[P_I, P_I]$ and $Q_I$. Moreover, both $[P_I, P_I]$ and $Q_I$ are normalized by $P_I$ so their Lie algebras in particular admit $T$-representations. However, the Lie algebra of $[P_I, P_I]$ contains all of $\mathfrak{n} = [\mathfrak{b}, \mathfrak{b}]$ and each $h_{\beta} := [e_{\beta}, f_{\beta}]$ and each $f_{\beta} = -2[h_{\beta}, f_{\beta}]$ for every simple coroot $\beta^{\vee} \in I$ where $e_{\beta} \in \LN_{\beta}$ and $f_{\beta} \in \LN^-_{\beta}$ yield an $\LSL_2$-triple. Therefore the tangent spaces of both groups of \cref{Commutator of Parabolic is Stabilizer of Vector} are equal, as desired.
\end{proof}

We can use similar methods to the proof of \cref{Commutator of Parabolic is Stabilizer of Vector} to obtain an explicit description of the $G$-stabilizer of any vector $\vec{v} \in E_J$. By the semidirect product decomposition of \labelcref{Semidirect Product Decomposition for Parabolics}, to compute the stabilizer of any vector it suffices to compute the stabilizer of $\vec{v}_{\tilde{J}} := \sum_{j \in \tilde{J}}\vec{v}_j$ for any fixed subset $\tilde{J} \subseteq J$. We describe this in terms of $\tilde{I} := \Delta^{\vee}\setminus \tilde{J}$:

\begin{Corollary}\label{Stabilizer of Generic Vector is Intersection of Stabilizer of Fixed Vectors}
    The stabilizer $S_{\vec{v}_{\tilde{J}}}$ of $\vec{v}_{\tilde{J}}$ is $Q_{\tilde{I}}$. 
\end{Corollary}

\begin{proof}
Letting $S_{\vec{v}_{j}}$ denote the stabilizer of $\vec{v}_j$, we have \[S_{\vec{v}_{\tilde{J}}} = \cap_{j \in \tilde{J}} S_{\vec{v}_{j}}\] since $V_{\tilde{J}} = \oplus_{j \in \tilde{J}} V_j$. 
Since the stabilizer of a vector is contained in the stabilizer of the line it spans, the stabilizer of $\vec{v}_j$ in $G$ is $Q_{\Delta^{\vee}\setminus\{j\}}$ for any $j \in \Delta^{\vee}$; in other words, $S_{\vec{v}_j} = Q_{\Delta^{\vee}\setminus \{j\}}$. Therefore \[\cap_{j \in \tilde{J}}S_{\vec{v}_j} = \cap_{j \in \tilde{J}}Q_{\Delta^{\vee}\setminus\{j\}} \subseteq \cap_{j \in \tilde{J}}P_{\Delta^{\vee}\setminus\{j\}} = P_I\] and so in particular we see that \[P_I \cap \bigcap_{j \in \tilde{J}}S_{\vec{v}_{\tilde{j}}} = \cap_{j \in \tilde{J}}S_{\vec{v}_j}.\] However, the kernel $Q_{\tilde{I}}$ of the representation $\rho_{\tilde{I}}$ is evidently the intersection \[\cap_{j \in \tilde{J}} (S_{\vec{v}_{\tilde{j}}} \cap P_I) = P_I \cap \bigcap_{j \in \tilde{J}}S_{\vec{v}_{\tilde{j}}}\] so our claim follows from combining these above equalities. 
\end{proof}

\begin{Remark}
    We thank the referee for suggesting the above proof of \cref{Stabilizer of Generic Vector is Intersection of Stabilizer of Fixed Vectors}, which is more simple than the proof originally given by the author. 
\end{Remark}

\subsubsection{Stratification of Affine Closure of $G/[P_I, P_I]$} From \cref{Commutator of Parabolic is Stabilizer of Vector}, we can now show the following claim, which when $I = \emptyset$ recovers the first sentence of \cref{Stratification of Affine Closure of Basic Affine Space by Partial Basic Affine Spaces Proposition}: 
\begin{Proposition}\label{G Mod Commutator of Parabolic Is Qaff With Stratification}
     The variety $G/[P_I, P_I]$ is quasi-affine. Moreover, we have a $G \times T$-equivariant stratification of its affine closure \begin{equation}
    \label{Stratification of Affine Closure of COMMUTATOR OF PARABOLIC by Partial Basic Affine Spaces}  \overline{G/[P_I, P_I]} =  \bigcup_{K \supseteq I}G/[P_K, P_K]
    \end{equation} such that $G/[P_{K'}, P_{K'}]$ is in the closure of $G/[P_K, P_K]$ if and only if $K \subseteq K'$. 
\end{Proposition}

\begin{proof}
   Notice that the representation morphism $\rho: G \times V_J \xrightarrow{} V_J$ induces a map $\tilde{\rho}: G \mathop{\times}\limits^{P_I} V_J \to V_J$ and so we obtain an induced map \[\tilde{\rho}|_{E_J}: G \mathop{\times}\limits^{P_I} E_J \to V_J\] since $E_J \subseteq V_J$ is $P_I$-stable. We first claim that $\tilde{\rho}|_{E_J}$ is proper. To see this, notice that, if $\phi$ denotes the isomorphism $G \mathop{\times}\limits^{P_I} V_J \xrightarrow{\sim} G/P_I \times V_J$ given by $\phi(g, v) := (gP, gv)$, we have $\text{proj}_{V_J}\phi = \tilde{\rho}$. This shows that $\tilde{\rho}$ is proper, and, since $\tilde{\rho}|_{E_J}$ is the restriction to the closed subvariety $G \mathop{\times}\limits^{P_I} E_J$, we see $\tilde{\rho}|_{E_J}$ is proper as well. In particular, the image $X_I$ of $\tilde{\rho}|_{E_J}$ is closed. Therefore, $X_I$ is affine and every closed point of $X_I$ can be written as $g\vec{v}$ for some closed point $\vec{v} \in E_J(k)$. 
    
    Now, notice the $P_I$-orbits of $E_J$ are equivalently given by the $\mathbb{G}_m^J$-orbits on $E_J$. Since we have a $\mathbb{G}_m^J$-equivariant isomorphism $E_J \cong \prod_{j \in J}L_j$ we may explicitly compute the $\mathbb{G}_m^J$-orbits on this space and we see that the $\mathbb{G}_m^J$-orbits on $E_J$ are precisely of the form \[\mathbb{O}_{K'} := \prod_{j \in K'}\{0\} \times \prod_{j \in J\setminus K'} L_j\setminus \{0\}\] for some $K' \subseteq J$, such that $\mathbb{O}_{K'}$ lies in the closure of $\mathbb{O}_{K''}$ if and only if $K' \supseteq K''$. Now, using \labelcref{Semidirect Product Decomposition for Parabolics}, we see that $G\mathbb{O}_{K'}$ is equivalently the $G$-orbit of $\sum_{j \in J\setminus K'}\vec{v}_{j}$ and therefore $G\mathbb{O}_{K'} \cong G/Q_{I \cup K'}$ by \cref{Stabilizer of Generic Vector is Intersection of Stabilizer of Fixed Vectors}. We therefore have $G\mathbb{O}_{K'} = G/[P_{I \cup K'}, P_{I \cup K'}]$ by \cref{Commutator of Parabolic is Stabilizer of Vector}, which gives the stratification \begin{equation}\label{Stratification of Image of Map}X_I = \bigcup_{K \supseteq I}G/[P_K, P_K]\end{equation} with the closure relations as in \cref{G Mod Commutator of Parabolic Is Qaff With Stratification}. 
    
    The variety $X_I$ is normal by \cite[Proposition 1]{KempfOnTheCollapsingofHomogeneousBundles}. We claim that the complement of $G/Q_I$ in $X_I$ is a codimension two subset. Indeed, this follows from the fact that for any coroot $j \in J$ the Lie algebra of any $Q_{I \cup \{j\}}$ contains a Levi factor $\SL^j_2$ of $P_{j}$ and so, comparing the $T$-weight spaces of the Lie algebras, we see that $\text{dim}(Q_{I \cup \{j\}}) \geq \text{dim}(Q_{I}) + 2$ and so $\text{dim}(G/Q_{I \cup \{j\}}) \leq \text{dim}(G/Q_{I}) - 2$. Therefore since $X_I$ is a normal affine variety (it is a closed subscheme of $E$), we have that $X_I$ is the affine closure of $G/Q_I$ and so our stratification \labelcref{Stratification of Image of Map} gives our claim. 
\end{proof}

\subsubsection{The Symplectic Vector Bundle}\label{The Symplectic Vector Bundle} Recall that in \cite[Section 2.1]{KazhdanLaumonGluingofPerverseSheavesandDiscreteSeriesRepresentation}, the authors choose an $\SL_2$-triple associated to a fixed simple root $\alpha$ and construct a certain rank two symplectic vector bundle $f_{\alpha}: V_{\alpha} \to G/Q_{\alpha}$ such that $V_{\alpha}$ admits an action of $G \times T$ for which the map $f_{\alpha}$ is equivariant for this action and such that the complement of the zero section of $V_{\alpha}$ can be identified with $G/N$. In particular, the variety $T^*(V_{\alpha})$ has an endomorphism given by the \textit{symplectic Fourier transform} \[S_{\alpha}: T^*(V_{\alpha}) \to T^*(V_{\alpha})\] see for example \cite[Appendix B]{GinzburgRicheDifferentialOperatorsOnBasicAffineSpaceandtheAffineGrassmannian} where, in the notation of \cite{GinzburgRicheDifferentialOperatorsOnBasicAffineSpaceandtheAffineGrassmannian}, we specialize to $\hbar = 0$. This notation is justified as the composite endomorphism (read left to right) \[\ringOfFunctionsForCOTANGENTBUNDLEOfBasicAffineSpace \xleftarrow{\sim} \mathcal{O}(T^*(V_{\alpha})) \xrightarrow{\sim} \mathcal{O}(T^*(V_{\alpha})) \xrightarrow{\sim} \ringOfFunctionsForCOTANGENTBUNDLEOfBasicAffineSpace\] induced by the restriction maps (which are equivalences since the complement of $T^*(G/N)$ in $T^*(V_{\alpha})$ has codimension two) and pulling back by $S_{\alpha}$ is, by definition, the automorphism $s_{\alpha}$ of $R$ given by the Gelfand-Graev action, see \cite{GinzburgRicheDifferentialOperatorsOnBasicAffineSpaceandtheAffineGrassmannian}.

\begin{Corollary}\label{KL Vector Bundle is Quasiaffine}
    The vector bundle $V_{\alpha}$ is quasi-affine. 
\end{Corollary}

\begin{proof}
    This is a standard argument (see, for example, the proof of \cite[Lemma 3.21]{GannonUniversal}) but since the argument is short we repeat it for the convenience of the reader. The morphism $f_{\alpha}$ is affine as $V_{\alpha}$ is a vector bundle, and therefore quasi-affine. Since the terminal map from $G/[P_{\alpha}, P_{\alpha}]$ is quasi-affine by \cref{G Mod Commutator of Parabolic Is Qaff With Stratification}, our claim follows from the fact that composition of quasi-affine morphisms is quasi-affine \cite[Lemma 29.13.4]{StacksProject}.
\end{proof}

\begin{Corollary}\label{Valpha is the expected two strata}
    In the stratification \labelcref{Stratification of Affine Closure of Basic Affine Space by Partial Basic Affine Spaces}, we can naturally identify the open subscheme $G/N \cup G/[P_{\alpha}, P_{\alpha}]$ with $V_{\alpha}$. 
\end{Corollary}
\begin{proof}
Recall that the complement of $G/N \subseteq V_{\alpha}$ has codimension two, since at the reduced level it can be identified with the scheme theoretic image of the zero section $G/Q_{\alpha} \to V_{\alpha}$ and $V_{\alpha}$ is a rank two vector bundle. Therefore since $V_{\alpha}$ is smooth and in particular normal we therefore see that the restriction map gives an equivalence $\mathcal{O}(V_{\alpha}) \xrightarrow{\sim} A$. Since $V_{\alpha}$ is quasi-affine by \cref{KL Vector Bundle is Quasiaffine}, the affinization map for $V_{\alpha}$ is an open embedding, and so we have an open embedding $V_{\alpha} \subseteq \overline{G/N}$ given by the composite (read left to right) \[V_{\alpha} \subseteq \text{Spec}(\mathcal{O}(V_{\alpha})) \xleftarrow{\sim} \affineClosureofBasicAffineSpace\] whose restriction to the open $G$-orbit $G/N$ the open embedding $G/N \subseteq \affineClosureofBasicAffineSpace$. Therefore the claim follows by comparing the stratification of $V_{\alpha} = G/N \cup G/[P_{\alpha}, P_{\alpha}]$ into the orbits of the action of $G \times T$ to the stratification of \cref{G Mod Commutator of Parabolic Is Qaff With Stratification} with $I = \emptyset$.
\end{proof}

Since $V_{\alpha}$ is smooth, we in particular see the following, which completes the proof of \cref{Stratification of Affine Closure of Basic Affine Space by Partial Basic Affine Spaces Proposition}:

\begin{Corollary}\label{Smooth Locus Has What You Expected For Simply Connected}
    The smooth locus of $\affineClosureofBasicAffineSpace$ contains $G/N$ and $G/[P_{\alpha}, P_{\alpha}]$ for any simple root $\alpha$.
\end{Corollary}

\subsection{Irreducible Elements of Functions on $T^*(G/N)$} 
In this section, we use the Gelfand-Graev action to compute some irreducible elements of $\ringOfFunctionsForCOTANGENTBUNDLEOfBasicAffineSpace$: 


\begin{Lemma}\label{The GK Generators of ringOfFunctionsForCOTANGENTBUNDLEOfBasicAffineSpace Are Irreducible}
    For all fundamental weights $\omega_i$, $w \in W$, and nonzero $z \in \ringOfFunctionsForBasicAffineSpace_{\omega_{i}}$, the element  $w(z) \in \ringOfFunctionsForCOTANGENTBUNDLEOfBasicAffineSpace$ is irreducible. 
\end{Lemma}

\begin{proof}
    It suffices to show this in the case when $w = 1$ since any $w \in W$ gives a ring automorphism which in particular preserves irreducibility. Now, if $z = ab$ for some $a, b \in \ringOfFunctionsForCOTANGENTBUNDLEOfBasicAffineSpace$, then by the $\mathbb{Z}^{\geq 0}$-grading given by the usual $\mathbb{G}_m$-action on the cotangent bundle, we see that $a, b \in A$. However, in $A$, we may identify the $\characterlatticeforT$-grading with a $(\mathbb{Z}^{\geq 0})^r$ grading using the fundamental weights. Then the irreducibility of $z$ then follows as $A_0 = k$, which implies that any nonzero element of $A_{v}$ for $v \in (\mathbb{Z}^{\geq 0})^r$ of length one must be irreducible. 
\end{proof}

\subsection{Torus Stabilizers and Projections}\label{Torus Stabilizers and Projections Subsection} For any global function $f$ on some scheme $Y$, we let $D(f)$ denote the complement of the vanishing locus of $f$, and, for any subset of global functions $F$, we let $D(F) := \cup_{f \in F}D(f)$. We now prove the following proposition, which informally says that for any $\mathfrak{p} \in D(\ringOfFunctionsForCOTANGENTBUNDLEOfBasicAffineSpace_{\lambda})$, there exists some $w \in W$ such that $\projectionFromAffineClosureofCotangentBundleToAffineClosureofSpace_w(\mathfrak{p})$ lies in an open locus of $\affineClosureofBasicAffineSpace$ determined by root hyperplanes which do not contain $\lambda$.

\begin{Proposition}\label{If You Have a Homogeneous Function Nonvanishing at Some Point with a Homogeneous Grading Off Some Walls Then You Have Control Over One of Its Projections}
    Assume $\lambda \in \characterlatticeforT$ and $w \in W$ such that $w\lambda$ lies in the closure of the dominant Weyl chamber. Write $w\lambda = \sum_i n_i\omega_i$ with $n_i \in \mathbb{Z}^{\geq 0}$, and let $S_{w\lambda}$ denote the subset of fundamental weights $\omega_i$ for which $n_i \neq 0$. Then $\projectionFromAffineClosureofCotangentBundleToAffineClosureofSpace_w(D(R_{\lambda}))$ maps into the open subscheme $\cap_{\omega_i \in S_{w\lambda}}D(A_{\omega_i})$. In particular, if $\lambda$ is regular then $\projectionFromAffineClosureofCotangentBundleToAffineClosureofSpace_w(D(R_{\lambda}))$ maps into $G/N$. 
\end{Proposition}

\begin{proof}[Proof of \cref{If You Have a Homogeneous Function Nonvanishing at Some Point with a Homogeneous Grading Off Some Walls Then You Have Control Over One of Its Projections}]
    By \cref{affineClosureOfCotangentBundleofBasicAffineSpace is a Closed Subscheme of |W| Many Products of G/N times LGd timesLTdsslash W LTD and Multiplication Map Is Surjective on Graded Pieces}, we see that the multiplication map \[\Sym(\LG) \otimes_{Z\LG} \Sym(\LT) \otimes \ringOfFunctionsForBasicAffineSpace_{w\lambda}  \to \ringOfFunctionsForCOTANGENTBUNDLEOfBasicAffineSpace_{w\lambda}\] is surjective. As $\ringOfFunctionsForBasicAffineSpace$ is generated by the union of the $\ringOfFunctionsForBasicAffineSpace_{\omega}$ where $\omega$ varies over the fundamental weights, we therefore in particular obtain the multiplication map \begin{equation}\label{Surjective Multiplication Map Pulled Back to Appropriate Fundamental Weights} \Sym(\LG) \otimes_{Z\LG} \Sym(\LT) \otimes \bigotimes_i A_{\omega_i}^{\otimes n_i} \to \ringOfFunctionsForCOTANGENTBUNDLEOfBasicAffineSpace_{w\lambda}\end{equation} is surjective. 
    
    Now fix some $\mathfrak{p} \in \affineClosureOfCotangentBundleofBasicAffineSpace$ such that $\ringOfFunctionsForCOTANGENTBUNDLEOfBasicAffineSpace_{\lambda} \not\subseteq \mathfrak{p}$, so that $\ringOfFunctionsForCOTANGENTBUNDLEOfBasicAffineSpace_{w\lambda} \not\subseteq w\mathfrak{p}$. 
    Let $I$ denote the ideal of $R$ generated by $R_{w\lambda}$ and, for a fixed $\omega_i \in S_{w\lambda}$, let $J_i$ denote the ideal of $\ringOfFunctionsForCOTANGENTBUNDLEOfBasicAffineSpace$ generated by $A_{\omega_{i}}$. The fact that the multiplication map of \labelcref{Surjective Multiplication Map Pulled Back to Appropriate Fundamental Weights} is surjective implies that $R_{w\lambda} \subseteq J_i$ and so $I \subseteq J_i$. In particular, we see that $A_{\omega_{i}} \not\subseteq w\mathfrak{p}$. Thus $\projectionFromAffineClosureofCotangentBundleToAffineClosureofSpace_w(\mathfrak{p}) = A \cap w\mathfrak{p}$ does not contain all of $A_{\omega_{i}}$ for any $\omega_i \in S_{w\lambda}$, as desired.     
\end{proof}

From this, we derive the following:

\begin{Corollary}\label{If you don't vanish on a set of function whose T-weights span LTd You Project into Some G/N Up To W-Action}
    Assume $\lambda_1, ..., \lambda_q \in \characterlatticeforT$ span the vector space $\LTd(\mathbb{Q})$. Then there exists some $w \in W$ such that $\projectionFromAffineClosureofCotangentBundleToAffineClosureofSpace_w$ maps the set $\cap_i D(R_{\lambda_{i}})$ into $G/N$. In particular $T$ acts with no stabilizer on any point in $\cap_i D(R_{\lambda_{i}})$. 
\end{Corollary}

We show \cref{If you don't vanish on a set of function whose T-weights span LTd You Project into Some G/N Up To W-Action} after proving the following Lemma:  

\begin{Lemma}\label{In a spanning cone you can always find a spanning lattice point away from some hyperplanes}
    Assume $L \subseteq \mathbb{Q}^d$ is some full rank lattice and choose some basis $S := \{\vec{v}_1, ..., \vec{v}_d\} \subseteq L$ of $\mathbb{Q}^d$. Denote by $C$ the $\mathbb{R}^{> 0}$-span of $S$, i.e. \[C := \{\sum_{i = 1}^d\alpha_i\vec{v}_i : \alpha_i \in \mathbb{R}^{> 0}\}\] and assume $\mathcal{Z}$ is some closed subset of $\mathbb{R}^d$ which does not contain $C$ and which is closed under scaling by any positive real number. Then the $\mathbb{Z}^{> 0}$-span of $S$ contains a point of $L$ not in $\mathcal{Z}$. 
\end{Lemma}

\begin{proof}
    Since $C$ is open, $C\cap \mathcal{Z}^c$ is open, and so in particular $C\cap \mathcal{Z}^c$ contains an open ball of some positive radius since by assumption it is nonempty. Since $C\cap \mathcal{Z}^c$ is closed under the scaling of any positive real number, for any $r \in \mathbb{R}^{> 0}$, $C\cap \mathcal{Z}^c$ contains an open ball of radius $r$. However, for any full rank lattice, there exists some $M$ such that all points in $\mathbb{R}^d$ are distance at most $M$ from a point on that lattice. Choosing $r > M$ we see that there is an element of $x \in L \cap \mathcal{Z}^c$ in the $\mathbb{R}^{> 0}$-span of $S$. As $S$ is a basis and $x$ in particular lies in $\mathbb{Q}^d$, we see that $x$ lies in the $\mathbb{Q}^{> 0}$-span of $S$. There exists some positive integer $N$ such that $Nx$ therefore is a $\mathbb{Z}^{> 0}$-linear combination of the $\vec{v}_i$, as desired.  
 \end{proof}

\begin{proof}[Proof of \cref{If you don't vanish on a set of function whose T-weights span LTd You Project into Some G/N Up To W-Action}]
    Assume $\mathfrak{p} \in \cap_i D(R_{\lambda_i})$. Let $L := \characterlatticeforT$ and choose some subset $S$ of the $\lambda_i$ such that $S$ is a basis of $\characterlatticeforT \otimes_{\mathbb{Z}} \mathbb{Q}$. If $\mathcal{Z}$ denotes the union of root hyperplanes, we may apply \cref{In a spanning cone you can always find a spanning lattice point away from some hyperplanes} to show there is some $\lambda \in \characterlatticeforT$ which lies in the interior of some Weyl chamber and which is a $\mathbb{Z}^{> 0}$-linear combination of the elements of $S$. Since $\mathfrak{p} \in D(R_{\lambda_i})$ for every $i$, we see $\mathfrak{p} \in D(R_{\lambda})$. Choose the (unique) $w \in W$ which takes $\lambda$ to the dominant Weyl chamber. Then since $w\lambda$ is regular, by \cref{If You Have a Homogeneous Function Nonvanishing at Some Point with a Homogeneous Grading Off Some Walls Then You Have Control Over One of Its Projections} we see that $\projectionFromAffineClosureofCotangentBundleToAffineClosureofSpace(w\mathfrak{p})$ maps to $G/N$. In particular, $\projectionFromAffineClosureofCotangentBundleToAffineClosureofSpace(w\mathfrak{p})$ has no $T$-stabilizer, and thus neither does $\mathfrak{p}$ itself, since $\projectionFromAffineClosureofCotangentBundleToAffineClosureofSpace$ is $T$-equivariant. 
 \end{proof}
From this we may also derive a similar result for the locus of points whose $T$-stabilizer has dimension one: 

    \begin{Corollary}\label{If You're Stabilized by a Dimension One Subtorus Then Up to W-Action You Project to G/Commutator of Palpha}
       Assume that the $T$-stabilizer of some point $\mathfrak{p} \in \affineClosureOfCotangentBundleofBasicAffineSpace$ has dimension one. Then there exists some $w \in W$ and some simple root $\alpha$ such that $\projectionFromAffineClosureofCotangentBundleToAffineClosureofSpace(w\mathfrak{p}) \in G/[P_{\alpha}, P_{\alpha}]$. In particular, the $T$-stabilizer of $\mathfrak{p}$ is $w\mathbb{G}_m^{\alpha}w^{-1}$.
    \end{Corollary}

    \begin{proof}
        For such a point $\mathfrak{p}$, we define the set $\mathcal{D} := \{\lambda \in \characterlatticeforT : \mathfrak{p} \in D(\ringOfFunctionsForCOTANGENTBUNDLEOfBasicAffineSpace_{\lambda})\}$. Using \labelcref{characterlatticeforT tensors to give Lie Algebra}, we may view $\mathcal{D}$ as a subset of $\LTd(\mathbb{Q})$ and let $V_{\mathbb{Q}}$ denote the span of $\mathcal{D}$ in $\LTd(\mathbb{Q})$. 
        
        We claim that the dimension of $V_{\mathbb{Q}}$ is at least $\text{dim}(\LTd) - 1$. To see this, assume the dimension of $V_{\mathbb{Q}}$ was less than $\text{dim}(\LTd) - 1$. In this case, there would be two linearly independent elements $\nu_1, \nu_2 \in \LT(\mathbb{Q})$ such that $\nu_1(s) = 0 = \nu_2(s)$ for any $s \in \mathcal{D}$. Here, we use the finite dimensionality of $\LT(\mathbb{Q})$ to canonically identify it with the vector space dual of $\LTd(\mathbb{Q})$. Now using the isomorphism \labelcref{cocharacterlatticeforT tensors to give Lie Algebra} we see that we may replace $\nu_1$ and $\nu_2$ by a positive integer multiple if necessary and additionally assume that both $\nu_1$ and $\nu_2$ are in $X_{\bullet}(T)$. By the definition of $\mathcal{D}$, we have that the image of each map $\nu_i: \mathbb{G}_m \to T$ stabilizes $\mathfrak{p}$. However, the fact that $\nu_1$ and $\nu_2$ form a linearly independent set implies that the subgroup generated by the images of these $\nu_i$ has dimension two, violating our assumption on the dimension of the stabilizer of $\mathfrak{p}$. 
        
        Let $V_{\mathbb{R}}$ denote the $\mathbb{R}$-span of the elements of $\mathcal{D}$, let $\mathcal{Z}'$ denote the union of every intersection of two distinct root hyperplanes, and let $\mathcal{Z} := V_{\mathbb{R}} \cap \mathcal{Z}'$. We also choose a $\mathbb{Q}$-basis $S$ of $V_{\mathbb{Q}} \subseteq V_{\mathbb{R}}$ contained in $\mathcal{D} \subseteq V_{\mathbb{Q}}$ and let $L$ denote the lattice generated by $S$. Since the dimension of $V_{\mathbb{Q}}$ is at least $\text{dim}(\LTd) - 1$, we see that $\mathcal{Z}$ does not contain the $\mathbb{R}^{\geq 0}$-span of $S$, and so we may apply \cref{In a spanning cone you can always find a spanning lattice point away from some hyperplanes} to see that the $\mathbb{Z}^{\geq 0}$-span of the elements of $S$ contains some $\lambda \in \characterlatticeforT$ such that $\lambda$ lies on at most one root hyperplane. Since the $\mathbb{Z}^{\geq 0}$-span of $S$ is contained in $\mathcal{D}$ as $\mathcal{D}$ is its own $\mathbb{Z}^{\geq 0}$-span, we see that $\mathfrak{p} \in D(\ringOfFunctionsForCOTANGENTBUNDLEOfBasicAffineSpace_{\lambda})$. 
        
        Choose some $w \in W$ such that $w\lambda$ lies on a root hyperplane cut out by at most one simple coroot. 
        By \cref{If You Have a Homogeneous Function Nonvanishing at Some Point with a Homogeneous Grading Off Some Walls Then You Have Control Over One of Its Projections}, we see that if $w\lambda$ is contained in no root hyperplanes then $\projectionFromAffineClosureofCotangentBundleToAffineClosureofSpace(w\mathfrak{p})$ projects into $G/N$ and so in particular the $T$-action on $\mathfrak{p}$ is free. Thus $w\lambda$ is contained in exactly one root hyperplane cut out by the vanishing of a simple coroot $\alpha^{\vee}$. Therefore by \cref{If You Have a Homogeneous Function Nonvanishing at Some Point with a Homogeneous Grading Off Some Walls Then You Have Control Over One of Its Projections} $\projectionFromAffineClosureofCotangentBundleToAffineClosureofSpace(w\mathfrak{p}) \in G/[P_{\alpha}, P_{\alpha}]$. The latter claim follows from the fact that $\projectionFromAffineClosureofCotangentBundleToAffineClosureofSpace$ is compatible with the $T$-action, which shows that the $T$-stabilizer of $w\mathfrak{p}$ is a closed subgroup scheme of $\mathbb{G}_m^{\alpha}$. 
    \end{proof}

        \begin{Corollary}\label{T Stabilizer Having Dimension leq 1 Implies Smooth Point}
        If the right $T$-stabilizer of some point $\mathfrak{p} \in \affineClosureOfCotangentBundleofBasicAffineSpace$ has dimension zero or one, then $\mathfrak{p}$ is a smooth point.
    \end{Corollary}

    \begin{proof}
        If the $T$-stabilizer of $\mathfrak{p} \in \affineClosureOfCotangentBundleofBasicAffineSpace$ has dimension zero, then we in particular see that the set of $w\omega_i$ for which $\mathfrak{p} \in D(R_{w\omega_i})$ spans the rational points of $\LTd$, since otherwise they would be contained in some hyperplane cut out by some $\alpha \in \LT(\mathbb{Q})$ and thus fixed by some subtorus. Therefore by \cref{If you don't vanish on a set of function whose T-weights span LTd You Project into Some G/N Up To W-Action}       
        some element of the $W$-orbit of $\mathfrak{p}$ projects to $G/N$ under $\projectionFromAffineClosureofCotangentBundleToAffineClosureofSpace$. 
        If the $T$-stabilizer of $\mathfrak{p}$ has dimension one, by \cref{If You're Stabilized by a Dimension One Subtorus Then Up to W-Action You Project to G/Commutator of Palpha} some point in its $W$-orbit projects to some point of $G/[P_{\alpha}, P_{\alpha}]$ under $\projectionFromAffineClosureofCotangentBundleToAffineClosureofSpace$. In either case, $\pi_w(\mathfrak{p})$ lies in $G/N \cup G/[P_{\alpha}, P_{\alpha}]$ for some $w$.  Since $G/N \cup G/[P_{\alpha}, P_{\alpha}] \subseteq \smoothLocusOfAffineclosureBasicAffineSpace$ by our simply connectedness assumption, we see that $\mathfrak{p}$ is smooth by \cref{Cotangent Bundle of Smooth Locus is Fiber of affineClosureOfCotangentBundleofBasicAffineSpace Over Smooth Locus of G/N-bar}. 
    \end{proof}
\subsection{Codimension of Singular Locus of Affine Closure of $T^*(G/N)$} \label{Codimension of Singular Locus of affineClosureOfCotangentBundleofBasicAffineSpace Is At Least Four Subsection}
We now use the results of \cref{Torus Stabilizers and Projections Subsection} to prove the following, which is a key ingredient in our proof of \cref{Affine Closure of Cotangent Bundle of Basic Affine Space Has Symplectic Singularities}:

    \begin{Theorem}\label{Codimension of Singular Locus of affineClosureOfCotangentBundleofBasicAffineSpace is At Least Four}
        The singular locus of $\affineClosureOfCotangentBundleofBasicAffineSpace$ has codimension at least four.
    \end{Theorem}

By \cref{T Stabilizer Having Dimension leq 1 Implies Smooth Point}, \cref{Codimension of Singular Locus of affineClosureOfCotangentBundleofBasicAffineSpace is At Least Four} follows directly from the following proposition:

\begin{Proposition}\label{Locus of Points of affineClosureOfCotangentBundleofBasicAffineSpace whose T-stabilizer has dimension geq 2 has codimension at least four}
    The locus of points of $\affineClosureOfCotangentBundleofBasicAffineSpace$ whose $T$-stabilizer has dimension $\geq 2$ has codimension at least four. 
\end{Proposition}

With the exception of \cref{If You're Only Interested In the Simply Connected Case You Can Skip Ahead Remark}, the proof of \cref{Locus of Points of affineClosureOfCotangentBundleofBasicAffineSpace whose T-stabilizer has dimension geq 2 has codimension at least four} will occupy the remainder of \cref{Codimension of Singular Locus of affineClosureOfCotangentBundleofBasicAffineSpace Is At Least Four Subsection}. We first construct a stratification of $\affineClosureOfCotangentBundleofBasicAffineSpace$ by $T$-invariant locally closed subschemes which will also be used in \cref{Free Locus Has Codimension Four Subsection}. Let $F$ denote the set of fundamental weights. For any $(w, \omega) \in W \times F$, set 
\[A_{w, \omega} := w(A_{\omega}) \subseteq \ringOfFunctionsForCOTANGENTBUNDLEOfBasicAffineSpace.\] Notice that $A_{\omega} = A_{1, \omega}$ for any $\omega \in F$. By \cref{affineClosureOfCotangentBundleofBasicAffineSpace is a Closed Subscheme of |W| Many Products of G/N times LGd timesLTdsslash W LTD and Multiplication Map Is Surjective on Graded Pieces}, the sets \[\mathscr{S}_S := \cap_{(w, \omega) \in S}V(A_{w, \omega}) \cap \bigcap_{(w, \omega) \notin S}D(A_{w, \omega})\] give a stratification of $\affineClosureOfCotangentBundleofBasicAffineSpace$ by locally closed $T$-invariant subschemes, where $S$ varies over subsets of $W \times F$. Moreover, for a fixed $S$, any two closed points in $\mathscr{S}_S$ have the same (right) $T$-stabilizer $T_S$. Since there are finitely many such $\mathscr{S}_S$, \cref{Locus of Points of affineClosureOfCotangentBundleofBasicAffineSpace whose T-stabilizer has dimension geq 2 has codimension at least four} follows from the following proposition: 

\begin{Proposition}\label{Codimension of scrSsubS has Codimension at least Four}
    Assume that $S \subseteq W \times F$ such that $T_S$ has dimension at least two. Then $\mathscr{S}_S$ has codimension at least four in $\affineClosureOfCotangentBundleofBasicAffineSpace$. 
\end{Proposition}



\begin{proof}
        Fix a subset $S \subseteq W \times F$ such that $T_S$ has dimension at least two. Letting $S^0 := \ringOfFunctionsForCOTANGENTBUNDLEOfBasicAffineSpace$, we recursively construct $k$-algebras $S^1, S^2, S^3, S^4$ such that for every $i \in \{0,1,2,3\}$,  \begin{equation}\label{Dimension of Spec of Si Drops By At Least One For Each i} \text{dim}(\text{Spec}(S^{i + 1})) \leq \text{dim}(\text{Spec}(S^{i})) - 1    \end{equation} and such that by construction $\mathscr{S}_S$ is a locally closed subscheme of $\text{Spec}(S^{4})$. The existence of this construction automatically implies that $\mathscr{S}_S$ has codimension at least four. 

        Choose some fundamental weight $\omega$ such that $(1, \omega)\in S$. Such an $\omega$ exists since, if not, \[\mathscr{S}_S \subseteq \cap_{\omega}D(\ringOfFunctionsForBasicAffineSpace_{1, \omega}) = \cap_{\omega}D(\ringOfFunctionsForBasicAffineSpace_{\omega}) \subseteq \cap_{\omega}D(\ringOfFunctionsForCOTANGENTBUNDLEOfBasicAffineSpace_{\omega})\] where the intersections are taken over all fundamental weights, and therefore any point of $\mathscr{S}_S$ has no $T$-stabilizer by, for example, \cref{If you don't vanish on a set of function whose T-weights span LTd You Project into Some G/N Up To W-Action}. Choose some nonzero $a_1 \in \ringOfFunctionsForBasicAffineSpace_{\omega}$ and let $S^1 = \ringOfFunctionsForCOTANGENTBUNDLEOfBasicAffineSpace/a_1$. Note that since \[\Spec(S^1) \supseteq V(\ringOfFunctionsForBasicAffineSpace_{\omega}) = V(\ringOfFunctionsForBasicAffineSpace_{(1, \omega)})\] and $(1, \omega) \in S$, $\mathscr{S}_S$ is a locally closed subscheme of $\Spec(S^1)$. Moreover, since $\ringOfFunctionsForCOTANGENTBUNDLEOfBasicAffineSpace$ is a unique factorization domain and $a_1$ is an irreducible element of $\ringOfFunctionsForCOTANGENTBUNDLEOfBasicAffineSpace$ by \cref{The GK Generators of ringOfFunctionsForCOTANGENTBUNDLEOfBasicAffineSpace Are Irreducible}, $S^1$ is an integral domain. Since $a_1$ is a nonzero element in an integral domain, we also see that \labelcref{Dimension of Spec of Si Drops By At Least One For Each i} holds when $i = 0$. We also have that, since $a_1$ is homogeneous with respect to the $\characterlatticeforT$-grading, the ring $S^1$ admits a grading by $\characterlatticeforT$. 
        The grading on $S^1$ has the property that \begin{equation}\label{S1 Has A Nonzero Homogeneous Weight Space for Any Weight}
        S^1_{\lambda} \neq 0 \text{ for any }\lambda \in \characterlatticeforT 
        \end{equation} by \cref{affineClosureOfCotangentBundleofBasicAffineSpace is a Closed Subscheme of |W| Many Products of G/N times LGd timesLTdsslash W LTD and Multiplication Map Is Surjective on Graded Pieces} and the unique factorization of $\ringOfFunctionsForCOTANGENTBUNDLEOfBasicAffineSpace$ given by \cref{The Ring ringOfFunctionsForCOTANGENTBUNDLEOfBasicAffineSpace is a UFD}. 
        
        As $T_S$ has dimension at least two, we may choose two elements $\gamma_1, \gamma_2 \in X_{\bullet}(T_S)$ which are linearly independent in $\text{Lie}(T_S)(\mathbb{Q})$, and we denote by $T_1$, respectively $T_2$ the rank one subtori generated by $\gamma_1$, respectively $\gamma_2$. Define $S^2 := S^1_{\gamma_1 = 0}$ so that, by definition, $S^2$ is the subring of $S^1$ generated by those $S_{\lambda}^1$ such that $\langle \gamma_1, \lambda \rangle = 0$. Then since \labelcref{S1 Has A Nonzero Homogeneous Weight Space for Any Weight} holds, we may apply \cref{If You're a Graded Integral Finitely Generated Domain with a Nonzero Graded Element and You Take the GIT Quotient by Gm You Lose a Dimension} to see that \labelcref{Dimension of Spec of Si Drops By At Least One For Each i} holds when $i = 1$. Moreover, since $\mathscr{S}_S$ is a locally closed subscheme of $\Spec(S^1)^{T_1}$ and the quotient map identifies $\Spec(S^1)^{T_1}$ as a closed subscheme of $\Spec(S^2) = \Spec(S^1)\sslash T_1$ by \cref{For Scheme with Gm Action Fixed Points Are Closed Subscheme of GIT Quotient}, we see that we may view $\mathscr{S}_S$ as a locally closed subscheme of $\Spec(S^2)$. Similarly, define $S^3 = S^2_{\gamma_2 = 0}$. Since \labelcref{S1 Has A Nonzero Homogeneous Weight Space for Any Weight} holds for $\lambda$ which satisfy \[\langle \gamma_1, \lambda \rangle = 0 \neq \langle \gamma_2, \lambda \rangle\] we see that we may similarly apply \cref{If You're a Graded Integral Finitely Generated Domain with a Nonzero Graded Element and You Take the GIT Quotient by Gm You Lose a Dimension} to see that \labelcref{Dimension of Spec of Si Drops By At Least One For Each i} holds when $i = 2$, and, exactly as above, we may view $\mathscr{S}_S$ as a locally closed subscheme of $\Spec(S^3) = \Spec(S^2)\sslash T_2$. Moreover, $S^3$ is an integral domain as it is a subring of the integral domain $S^1$. 
        
        Finally, choose some $a_2 \in A_{\omega}$ such that $\{a_1, a_2\}$ is linearly independent and choose some nonzero $b \in \ringOfFunctionsForBasicAffineSpace_{w_0, w_0(-\omega)}$. If we set $S^4 := S^3/(a_2b)$ then, since $\omega \in S$, $\mathscr{S}_S$ (viewed as a locally closed subscheme of $\text{Spec}(S^3)$ via the quotient map as above) is contained $\Spec(S^4)$. We now claim that $a_2b$ is not zero in $S^3$. To see this, notice that if $a_2b = 0$ in $S^3 \subseteq S^1$ then there exists some $f \in \ringOfFunctionsForCOTANGENTBUNDLEOfBasicAffineSpace$ such that $a_1f = a_2b$ in $\ringOfFunctionsForCOTANGENTBUNDLEOfBasicAffineSpace$. However, since $a_1, a_2, b$ are irreducible in $\ringOfFunctionsForCOTANGENTBUNDLEOfBasicAffineSpace$ by \cref{The GK Generators of ringOfFunctionsForCOTANGENTBUNDLEOfBasicAffineSpace Are Irreducible}, this would violate the fact that $\ringOfFunctionsForCOTANGENTBUNDLEOfBasicAffineSpace$ is a unique factorization domain, i.e. \cref{The Ring ringOfFunctionsForCOTANGENTBUNDLEOfBasicAffineSpace is a UFD}. Since $a_2b$ is a nonzero element in the integral domain $S^3$, any irreducible component of $\text{Spec}(S^4)$ has codimension 1 in $\text{Spec}(S^3)$. 
\end{proof}
\begin{Remark}\label{If You're Only Interested In the Simply Connected Case You Can Skip Ahead Remark}
    The reader who is only interested in the proof of the main theorems when $G$ is simply connected may proceed directly to \cref{Symplectic Singularities of Normal Varieties With Small Singular Locus Subsection}, replacing the usage of \cref{Complement of Locus of Points on Which T Acts Freely for ARBITRARY REDUCTIVE GROUP Is at least Codimension Four} in final sentence of the proof of \cref{affineClosureOfCotangentBundleofBasicAffineSpace has symplectic singularities for ARBITRARY REDUCTIVE GROUP} with the usage of \cref{Codimension of Singular Locus of affineClosureOfCotangentBundleofBasicAffineSpace is At Least Four}.
\end{Remark}
\subsection{Free Locus Has Codimension Four}\label{Free Locus Has Codimension Four Subsection}
Using the notation introduced in \cref{Codimension of Singular Locus of affineClosureOfCotangentBundleofBasicAffineSpace Is At Least Four Subsection}, let $Z_i$ denote the closed subscheme $V(A_{\omega_{i}}) \cap V(A_{w_0, w_0(-\omega_i)})$, and let $U_i$ denote its open complement. In this section, we study properties of the open subscheme $Q := \cup_ww(\cap_{i = 1}^rU_i)$. We first give a more explicit description for $Q$, which can be compared to \cite[Proposition 5.1.4]{Gin}: 

\begin{Proposition}\label{T-Free Locus is W-Sweep of Pullback of G/N Which Is Also W-Sweep of Having pm an Aomegai Coordinate Not Vanishing for Every i}
    We have an equality\begin{equation}\label{W Orbit of Intersection of Nonvanishing of Coordinate Lines for Fundamental Weights is Union of W Many Preimages of Basic Affine Space}Q = \cup_w \projectionFromAffineClosureofCotangentBundleToAffineClosureofSpace^{-1}_w(G/N).    \end{equation} Moreover, $Q$ is the set of points of $\affineClosureOfCotangentBundleofBasicAffineSpace$ for which the (right) $T$ action is free and is the set of points of $\affineClosureOfCotangentBundleofBasicAffineSpace$ for which the $T$-stabilizer has dimension zero. In particular, any point of $\affineClosureOfCotangentBundleofBasicAffineSpace$ whose $T$-stabilizer has dimension zero has trivial stabilizer.
\end{Proposition}

\begin{proof}
    To show $\subseteq$ in \labelcref{W Orbit of Intersection of Nonvanishing of Coordinate Lines for Fundamental Weights is Union of W Many Preimages of Basic Affine Space}, by $W$-equivariance it suffices to show $\cap_{i = 1}^rU_i \subseteq  \cup_w\projectionFromAffineClosureofCotangentBundleToAffineClosureofSpace^{-1}_w(G/N)$. Choose some homogeneous function $f_i \in A_{\omega_{i}}\cup A_{w_0, w_0(-\omega_i)}$ and let $\lambda_i \in \{\pm \omega_i\}$ denote the degree of $f_i$. It further suffices to show that \[\cap_{i = 1}^rD(f_i) \subseteq \cup_w\projectionFromAffineClosureofCotangentBundleToAffineClosureofSpace^{-1}_w(G/N)\] which follows from \cref{If you don't vanish on a set of function whose T-weights span LTd You Project into Some G/N Up To W-Action}. Conversely, by $W$-equivariance we may show the containment $\projectionFromAffineClosureofCotangentBundleToAffineClosureofSpace^{-1}(G/N) \subseteq \cup_ww(\cap_{i = 1}^rU_i)$, but this follows from the fact that $\projectionFromAffineClosureofCotangentBundleToAffineClosureofSpace^{-1}(G/N) = \cap_{i = 1}^rD(A_{\omega_i}) \subseteq \cap_{i = 1}^rU_i$. 
    By the $T$-equivariance of $\projectionFromAffineClosureofCotangentBundleToAffineClosureofSpace$, we see that any point in $\affineClosureOfCotangentBundleofBasicAffineSpace$ which maps to $G/N$ under $\projectionFromAffineClosureofCotangentBundleToAffineClosureofSpace$ must itself have trivial (right) $T$-stabilizer. Now assume that the right $T$-action on $\affineClosureOfCotangentBundleofBasicAffineSpace$ for some point $\mathfrak{p}$ has dimension zero. Then for any $\gamma \in X_{\bullet}(T)$ there exists some $\lambda \in \characterlatticeforT$ and $f \in R_{\lambda}$ such that $\langle \lambda, \gamma \rangle \neq 0$ and $f$ does not vanish at $\mathfrak{p}$, as otherwise it would be stabilized by the one parameter subgroup $\gamma: \mathbb{G}_m \xhookrightarrow{} T$. 
    In particular, if we let $S$ denote the set of all $\lambda \in \characterlatticeforT$ such that $\ringOfFunctionsForCOTANGENTBUNDLEOfBasicAffineSpace_{\lambda}$ contains a function which does not vanish at $\mathfrak{p}$, then we see that the elements of $S$ span the real points of $\LTd$.  
    Thus we see that $\mathfrak{p} \in \projectionFromAffineClosureofCotangentBundleToAffineClosureofSpace^{-1}_w(G/N)$ for some $w \in W$ by \cref{If you don't vanish on a set of function whose T-weights span LTd You Project into Some G/N Up To W-Action} and, therefore by the above must also have trivial $T$-stabilizer.
    \end{proof}

From this, we can derive the codimension result on the complement of $Q$ (for $G$ simply connected) stated in the introduction, whose proof occupies the remainder of this entire subsection: 

\begin{Corollary}\label{Complement of Q Has Codimension Four in SIMPLY CONNECTED CASE}
    The complement of $Q$ has codimension at least four. 
\end{Corollary}

\newcommand{\baseOfValpha}{Y_{\alpha}}
For this remainder of this section, we fix a simple root $\alpha$, choose an $\SL_2$-triple for $\alpha$ and use the notation of \cref{The Symplectic Vector Bundle}. We also let $p_{\alpha}: T^*(V_{\alpha}) \to V_{\alpha}$ denote the projection map and set $p_{\alpha, s} := p_{\alpha} \circ S_{\alpha}$ where $S_{\alpha}: T^*(V_{\alpha}) \to T^*(V_{\alpha})$ is the symplectic Fourier transform. We also define $\baseOfValpha := G/Q_{\alpha}$ and view $\baseOfValpha$ as a closed subscheme of $V_{\alpha}$ via the zero section map.

\begin{Lemma}\label{The Complement of Things in Vs That Up to Simple Reflection Map to G mod N Is Cotangent Bundle of G mod Commutator of Palpha}
The scheme theoretic intersection \[Z_{\alpha} := p_{\alpha}^{-1}(\baseOfValpha) \cap p_{\alpha, s}^{-1}(\baseOfValpha)\] has codimension four in $T^*(V_{\alpha})$. 
\end{Lemma}

\begin{proof}
First, the symplectic Fourier transform on $V_{\alpha}$ is, by construction, an automorphism over $\baseOfValpha$. Therefore the map \[(p_{\alpha}, p_{\alpha, s}): T^*(V_{\alpha}) \to V_{\alpha} \times V_{\alpha}\] factors through the closed subscheme $V_{\alpha}\times_{\baseOfValpha} V_{\alpha}$. We wish to compute the dimension of the closed subscheme \[Z_{\alpha} = T^*(V_{\alpha}) \times_{V_{\alpha}\times_{\baseOfValpha} V_{\alpha}} \baseOfValpha\] of $T^*(V_{\alpha})$ where we regard $T^*(V_{\alpha})$ (and thus $Z_{\alpha}$) as a scheme over $\baseOfValpha$ in the natural way. It suffices to do this on an open cover of $V_{\alpha}$. 

Let $C$ denote the open $B$-orbit of $G/P_{\alpha}$, and let $\tilde{C}$ denote its inverse image under the quotient map $G/Q_{\alpha} \to G/P_{\alpha}$. Defining $U_0 := f_{\alpha}^{-1}(\tilde{C})$, the open subset $U := p_{\alpha}^{-1}(U_0)$ gives a nonempty open subscheme of $T^*(V_{\alpha})$. Moreover, since $f_{\alpha}$ and $p_{\alpha}$ are $G$-equivariant and the action of $G$ on $G/Q_{\alpha}$ is transitive, we see that $U$ and its $G(k)$-translates cover $T^*(V_{\alpha})$. It therefore suffices to show $U \cap Z_{\alpha}$ has codimension four in $U$ by $G(k)$-equivariance. 

The construction of $V_{\alpha}$ \cite[Section 2.1]{KazhdanLaumonGluingofPerverseSheavesandDiscreteSeriesRepresentation} gives a trivialization $U_0 = f_{\alpha}^{-1}(\tilde{C}) \cong \tilde{C} \times \mathbb{A}^2$ such that the formula \[(c, \begin{pmatrix}
    x_1 \\ x_2
\end{pmatrix}), (c, \begin{pmatrix}
    y_1 \\ y_2
\end{pmatrix}) \mapsto (c, x_1y_2 - y_1x_2)\] gives the symplectic form \[(\tilde{C} \times \mathbb{A}^2) \times_{\tilde{C}} (\tilde{C} \times \mathbb{A}^2) \to \tilde{C} \times \mathbb{G}_a\] on $V_{\alpha}$ restricted to this open subset. In particular, we obtain isomorphisms \begin{equation}\label{Isos Induced by Trivialization}U \cong T^*(\tilde{C} \times \mathbb{A}^2) \cong T^*(\tilde{C}) \times T^*(\mathbb{A}^2) \cong T^*(\tilde{C}) \times \mathbb{A}^2 \times \mathbb{A}^{2, *} \xrightarrow{\sim}T^*(\tilde{C}) \times \mathbb{A}^{2} \times \mathbb{A}^{2}\end{equation} where the final isomorphism is induced by the symplectic form. Under the composite identification obtained from reading \labelcref{Isos Induced by Trivialization} left to right one can directly follow the construction of the symplectic Fourier transform as in, for example \cite[Appendix B]{GinzburgRicheDifferentialOperatorsOnBasicAffineSpaceandtheAffineGrassmannian}, to see that it is given by the automorphism \[(z, \begin{pmatrix}x_1 \\ x_2 \end{pmatrix}, \begin{pmatrix}y_1 \\ y_2 \end{pmatrix}) \mapsto (z, \begin{pmatrix}y_2 \\ -y_1 \end{pmatrix}, \begin{pmatrix}-x_2 \\ x_1 \end{pmatrix})\] of $T^*(\tilde{C}) \times \mathbb{A}^{2} \times \mathbb{A}^{2}$. Therefore, using the above trivialization, we may identify $p_{\alpha}|_U$ with the map \[(z, \begin{pmatrix}x_1 \\ x_2 \end{pmatrix}, \begin{pmatrix}y_1 \\ y_2 \end{pmatrix}) \mapsto (\overline{z}, \begin{pmatrix}x_1 \\ x_2 \end{pmatrix})\]where $\overline{z}$ is the image of $p$ under the map $T^*(C) \to C$. We may similarly identify the map $p_{\alpha, s}|_U$ with \[(z, \begin{pmatrix}x_1 \\ x_2 \end{pmatrix}, \begin{pmatrix}y_1 \\ y_2 \end{pmatrix}) \mapsto (\overline{z}, \begin{pmatrix}y_2 \\ -y_1 \end{pmatrix})\] again using the identification induced by the composite of the isomorphisms of \labelcref{Isos Induced by Trivialization}. Therefore, via the isomorphisms of \labelcref{Isos Induced by Trivialization}, we can identify \[Z_{\alpha} \cap U \cong T^*(\tilde{C}) \times \{0\} \times \{0\}\] which evidently has codimension four in $U$. 
\end{proof}

\begin{proof}[Proof of \cref{Complement of Q Has Codimension Four in SIMPLY CONNECTED CASE}]
    We temporarily denote by $Z_Q$ the complement of $Q$ in $\affineClosureOfCotangentBundleofBasicAffineSpace$. Since we can identify $Q$ with the locus where $T$ acts with dimension zero stabilizers \cref{T-Free Locus is W-Sweep of Pullback of G/N Which Is Also W-Sweep of Having pm an Aomegai Coordinate Not Vanishing for Every i}, we can write $Z_Q = Z_Q^1 \cup Z_Q^{\geq 2}$ where $Z_Q^1$ denotes the open subscheme of $Z_Q$ consisting of those points whose right $T$-stabilizer has dimension one and $Z_Q^{\geq 2}$ is the closed subscheme consisting of those points whose right $T$-stabilizer has dimension at least two. By \cref{Locus of Points of affineClosureOfCotangentBundleofBasicAffineSpace whose T-stabilizer has dimension geq 2 has codimension at least four}, the codimension of $Z_Q^{\geq 2}$ is at least four, so it suffices to show that $Z_Q^1$ has codimension at least four. 

    By \cref{If You're Stabilized by a Dimension One Subtorus Then Up to W-Action You Project to G/Commutator of Palpha}, we may write $Z_Q^1$ as the union $Z_Q^1 = \cup_{w, \alpha} Z_{Q, w, \alpha}^1$, where $w$ varies over $W$, $\alpha$ varies over the simple roots, and $Z_{Q, w, \alpha}^1$ denotes the subset of points of $Z_{Q}^1$ such that $\projectionFromAffineClosureofCotangentBundleToAffineClosureofSpace_w(Z_{Q}^1) \in G/Q_{\alpha}$. In particular we have $w(Z_{Q, w, \alpha}^1) \subseteq \projectionFromAffineClosureofCotangentBundleToAffineClosureofSpace^{-1}(V_{\alpha}) \times_{V_{\alpha}} G/Q_{\alpha}$ by \cref{Valpha is the expected two strata}. Notice, however, that \[\projectionFromAffineClosureofCotangentBundleToAffineClosureofSpace^{-1}(V_{\alpha}) = \affineClosureOfCotangentBundleofBasicAffineSpace \times_{\affineClosureofBasicAffineSpace} V_{\alpha} \cong (\affineClosureOfCotangentBundleofBasicAffineSpace \times_{\affineClosureofBasicAffineSpace} \smoothLocusOfAffineclosureBasicAffineSpace) \times_{\smoothLocusOfAffineclosureBasicAffineSpace} V_{\alpha} \xleftarrow{\sim} T^*(\smoothLocusOfAffineclosureBasicAffineSpace) \times_{\smoothLocusOfAffineclosureBasicAffineSpace} V_{\alpha} \cong T^*(V_{\alpha})\] using the isomorphism given by \cref{Cotangent Bundle of Smooth Locus is Fiber of affineClosureOfCotangentBundleofBasicAffineSpace Over Smooth Locus of G/N-bar} and the containment $V_{\alpha} \subseteq \smoothLocusOfAffineclosureBasicAffineSpace$ given by \cref{Smooth Locus Has What You Expected For Simply Connected}. Therefore $w(Z_{Q, w, \alpha}^1) \subseteq T^*(V_{\alpha}) \times_{V_{\alpha}} \baseOfValpha$. We also have that $s_{\alpha}(w(Z_{Q, w, \alpha}^1)) \subseteq T^*(V_{\alpha}) \times_{V_{\alpha}} \baseOfValpha$ since if there was a point which was not contained in this closed subscheme, using the stratification $V_{\alpha} = G/N \cup \baseOfValpha$ of \cref{Valpha is the expected two strata}, we would see that $T$ acts freely on this point by \cref{T-Free Locus is W-Sweep of Pullback of G/N Which Is Also W-Sweep of Having pm an Aomegai Coordinate Not Vanishing for Every i}. Therefore $w(Z_{Q, w, \alpha}) \subseteq p_{\alpha}^{-1}(Y_{\alpha}) \cap p_{\alpha, s}^{-1}(Y_{\alpha})$ and so $w(Z_{Q, w, \alpha})$ has codimension at least four by \cref{The Complement of Things in Vs That Up to Simple Reflection Map to G mod N Is Cotangent Bundle of G mod Commutator of Palpha}. Therefore $Z_{Q, w, \alpha} = w^{-1}(w(Z_{Q, w, \alpha}))$ has codimension at least four as well. 
\end{proof}
\section{Proofs of Main Theorems}In this section, we record some consequences of our above computations. First, we extend results on $Q$ above to an arbitrary reductive group in \cref{Free Locus from Simply Connected Case Subsection}. We then show that one can derive the symplectic singularities of $\affineClosureOfCotangentBundleofBasicAffineSpace$ for all $G$ from this in \cref{Symplectic Singularities of Normal Varieties With Small Singular Locus Subsection} and then finish the proof of \cref{affineClosureOfCotangentBundleofBasicAffineSpace is Q-Factorial Terminal and is Factorial When G is Simply Connected} in \cref{Consequences of Main Theorem Subsection}. 
\subsection{Free Locus from Simply Connected Case}\label{Free Locus from Simply Connected Case Subsection}
\newcommand{\ringOfFunctionsForCOVERofCOTANGENTBUNDLEOfBasicAffineSpace}{\tilde{R}}
\newcommand{\affineClosureOfCOVERofCotangentBundleofBasicAffineSpace}{\overline{T^*(\tilde{G}/\tilde{N})}}
\newcommand{\quotientMapFromCoverOfCotangentBundletoCotangentBundle}{q}
\newcommand{\affinizationOfQuotientMapFromCoverOfCotangentBundletoCotangentBundle}{\overline{q}}
The group $G$ can be written as a quotient \begin{equation} 
    (G^{\text{sc}} \times \tilde{T})/Z \xrightarrow{\sim} G
\end{equation}
\noindent for $G^{\text{sc}}$ some simply connected semisimple group, $\tilde{T}$ some torus and $Z$ some closed finite central subgroup scheme of $\tilde{G} := G^{\text{sc}} \times \tilde{T}$. Let $\tilde{N}$ denote the unipotent radical of some Borel which projects into $B$ under the quotient map. The map $\tilde{G}/\tilde{N} \to G/N$ is a finite \'etale cover. From this, we see:

\begin{Lemma}
    The natural map induces an isomorphism \[\quotientMapFromCoverOfCotangentBundletoCotangentBundle: T^*(\tilde{G}/\tilde{N})/Z \xrightarrow{\sim} T^*(G/N)\] so that in particular if $\ringOfFunctionsForCOVERofCOTANGENTBUNDLEOfBasicAffineSpace$ denotes the ring of functions on $T^*(\tilde{G}/\tilde{N})$ we have $\ringOfFunctionsForCOTANGENTBUNDLEOfBasicAffineSpace \xrightarrow{\sim} \ringOfFunctionsForCOVERofCOTANGENTBUNDLEOfBasicAffineSpace^Z$. In other words, the map $\affinizationOfQuotientMapFromCoverOfCotangentBundletoCotangentBundle: \affineClosureOfCOVERofCotangentBundleofBasicAffineSpace \to \affineClosureOfCotangentBundleofBasicAffineSpace$ induces an isomorphism $\affinizationOfQuotientMapFromCoverOfCotangentBundletoCotangentBundle: \affineClosureOfCOVERofCotangentBundleofBasicAffineSpace\sslash Z \xrightarrow{\sim} \affineClosureOfCotangentBundleofBasicAffineSpace$. 
\end{Lemma}

As above, let $Q$ denote the locus of points of $\affineClosureOfCotangentBundleofBasicAffineSpace$ on which $T$ acts freely.

\begin{Theorem}\label{Complement of Locus of Points on Which T Acts Freely for ARBITRARY REDUCTIVE GROUP Is at least Codimension Four}
    The locus $Q$ is smooth. Moreover, the complement of $Q$ has codimension at least four. 
\end{Theorem}

\begin{proof}

We see that $\affinizationOfQuotientMapFromCoverOfCotangentBundletoCotangentBundle^{-1}(Q)$ is precisely the locus on which the maximal torus in $\tilde{G}$ acts with finite stabilizer group (since $\affinizationOfQuotientMapFromCoverOfCotangentBundletoCotangentBundle$ is a finite map). By \cref{T-Free Locus is W-Sweep of Pullback of G/N Which Is Also W-Sweep of Having pm an Aomegai Coordinate Not Vanishing for Every i} (which also evidently applies in the case where $G$ is the product of a simple simply connected group with some torus) we see that $\affinizationOfQuotientMapFromCoverOfCotangentBundletoCotangentBundle^{-1}(Q)$ is also the locus on which $\tilde{T}$ acts freely. In particular, the map $\affinizationOfQuotientMapFromCoverOfCotangentBundletoCotangentBundle^{-1}(Q) \to Q$ induces an isomorphism $\affinizationOfQuotientMapFromCoverOfCotangentBundletoCotangentBundle^{-1}(Q)/Z \xrightarrow{\sim} Q$. This shows that $Q$ is smooth. Moreover, the fact that $\affinizationOfQuotientMapFromCoverOfCotangentBundletoCotangentBundle$ is a surjection also shows that the map \[\affinizationOfQuotientMapFromCoverOfCotangentBundletoCotangentBundle|_{\affineClosureOfCOVERofCotangentBundleofBasicAffineSpace\setminus\affinizationOfQuotientMapFromCoverOfCotangentBundletoCotangentBundle^{-1}(Q)}: \affineClosureOfCOVERofCotangentBundleofBasicAffineSpace\setminus\affinizationOfQuotientMapFromCoverOfCotangentBundletoCotangentBundle^{-1}(Q) \to \affineClosureOfCotangentBundleofBasicAffineSpace\setminus Q\] is dominant. In particular, we see that the fact that the complement of $\affinizationOfQuotientMapFromCoverOfCotangentBundletoCotangentBundle^{-1}(Q)$ has codimension at least four (\cref{Complement of Q Has Codimension Four in SIMPLY CONNECTED CASE}) implies that the complement of $Q$ has codimension at least 4. 
\end{proof}

\subsection{Symplectic Singularities of Normal Varieties Via Codimension}\label{Symplectic Singularities of Normal Varieties With Small Singular Locus Subsection}
In this section, we prove \cref{Affine Closure of Cotangent Bundle of Basic Affine Space Has Symplectic Singularities}. When $G$ is semisimple, we have shown that $\affineClosureOfCotangentBundleofBasicAffineSpace$ admits a conical $\mathbb{G}_m$-action compatible with the Poisson bracket in \cref{Grading on Functions on T*(G/N) Subsection}. Therefore, it remains to show the following, which we prove for arbitrary reductive $G$:

\begin{Theorem}\label{affineClosureOfCotangentBundleofBasicAffineSpace has symplectic singularities for ARBITRARY REDUCTIVE GROUP}
    The variety $\affineClosureOfCotangentBundleofBasicAffineSpace$ has symplectic singularities.
\end{Theorem}

We recall the following standard lemma on extendability of differential forms as applied to the theory of symplectic singularities:

\begin{Lemma}\label{Normal Variety with Nondeg Form on Enough of Regular Locus with High Dimensional Singular Locus Has Symplectic Singularities} Assume $X$ is some irreducible variety. \begin{enumerate}
    \item Assume $Z$ is some closed subscheme of $X^{\text{reg}}$ whose codimension is larger than 1. Then if $\omega \in \Omega^2(X^{\text{reg}}\setminus Z)$ is some symplectic form, then $\omega$ extends to a symplectic form on $X^{\text{reg}}$. 
    \item If in addition $X$ is normal and $\text{codim}_X(X\setminus X^{\text{reg}}) \geq 4$, $X$ has symplectic singularities.
\end{enumerate}
\end{Lemma}

\begin{proof}
    A standard Hartog's lemma argument gives that $\omega$ extends to a symplectic form on $X^{\text{reg}}$, see, for example, \cite[Section 4, Remarque (3)]{BeauvilleVarietiesesKahleriennesDontLaPremiereClassedeChernestNulle}. Now, any normal irreducible variety with a nondegenerate $2$-form on the regular locus has symplectic singularities if and only if, for any resolution $p: Y \to X$ of singularities, the induced 2-form $p^*(\omega)$ extends to a 2-form on $Y$. However, since the codimension of the singular locus is larger than 3, this extension follows directly from the main theorem of \cite{FlennerExtendabilityofDifferentialFormsonNonisolatedSingularities}. 
\end{proof}

\begin{proof}[Proof of \cref{affineClosureOfCotangentBundleofBasicAffineSpace has symplectic singularities for ARBITRARY REDUCTIVE GROUP}]
    We check the hypotheses of \cref{Normal Variety with Nondeg Form on Enough of Regular Locus with High Dimensional Singular Locus Has Symplectic Singularities}. The fact that $\affineClosureOfCotangentBundleofBasicAffineSpace$ is normal follows from \cref{affineClosureOfCotangentBundleofBasicAffineSpace is normal and Q-factorial for all G and is factorial when G is simply connected}. The variety $T^*(G/N)$ is an open subscheme of $\affineClosureOfCotangentBundleofBasicAffineSpace$ whose complement has codimension at least two by \cref{Complement to QuasiAffine Noetherian Scheme Must Be of Codimension At Least Two}, so the symplectic form on $T^*(G/N)$ necessarily extends to a symplectic form on the smooth locus of $\affineClosureOfCotangentBundleofBasicAffineSpace$. Finally, the fact that  $\affineClosureOfCotangentBundleofBasicAffineSpace$ has a singular locus of codimension $\geq 4$ follows from \cref{Complement of Locus of Points on Which T Acts Freely for ARBITRARY REDUCTIVE GROUP Is at least Codimension Four}.
\end{proof}


\subsection{Consequences of Main Theorem}\label{Consequences of Main Theorem Subsection} We have seen that $\affineClosureOfCotangentBundleofBasicAffineSpace$ has symplectic singularities in \cref{affineClosureOfCotangentBundleofBasicAffineSpace has symplectic singularities for ARBITRARY REDUCTIVE GROUP} and that its singular locus has codimension at least four in \cref{Complement of Locus of Points on Which T Acts Freely for ARBITRARY REDUCTIVE GROUP Is at least Codimension Four}. Recall that all singular symplectic varieties whose singular locus has codimension at least four are terminal by the main result of \cite{NamikawaANoteOnSymplecticSingularities}. Therefore we immediately obtain the following result which, combined with \cref{affineClosureOfCotangentBundleofBasicAffineSpace is normal and Q-factorial for all G and is factorial when G is simply connected}, completes the proof of \cref{affineClosureOfCotangentBundleofBasicAffineSpace is Q-Factorial Terminal and is Factorial When G is Simply Connected}:

\begin{Corollary}\label{affineClosureOfCotangentBundleofBasicAffineSpace has Terminal Singularities}
    The variety $\affineClosureOfCotangentBundleofBasicAffineSpace$ has terminal singularities.
\end{Corollary}

We also claim that, if $G$ is semisimple and not a product of copies of $\SL_2$, that $\affineClosureOfCotangentBundleofBasicAffineSpace$ is singular: 

\begin{Proposition}
When $G$ is semisimple and not a product of copies of $\SL_2$, the cone point $0 \in \affineClosureOfCotangentBundleofBasicAffineSpace$ is singular.
\end{Proposition} 

\begin{proof}
    Notice that the ideal cutting out the image of the closed embedding given by the zero section \[\overline{z}: \affineClosureofBasicAffineSpace \xhookrightarrow{} \affineClosureOfCotangentBundleofBasicAffineSpace\] is homogeneous for both the torus action and the usual $\mathbb{G}_m$-action induced by scaling fibers on $T^*(G/N)$. Therefore, by \cref{Grading on Ring of Functions is (1) Admits Equivariant Projection and Moment Maps (2) Is Determined by T Weight and Height (3) is W-equivariant (4) is conical}(2), this ideal is also homogeneous for the conical $\mathbb{G}_m$-action. In particular, the cone point is contained in this closed subscheme. 
    
    It is well known that $\affineClosureofBasicAffineSpace$ is singular when $G$ is semisimple and not a product of copies of $\SL_2$--for example, this follows from the fact that the ring of differential operators on any smooth affine variety is generated by derivations \cite[Corollary 15.6]{McConnellRobsonNoncommutativeNoetherianRings} but the results of \cite{LevasseurStaffordDifferentialOperatorsandCohomologyGroupsontheBasicAffineSpace} show that this is not the case for such $G$. Since the singular locus of a scheme is closed and the singular locus of a scheme with a group action is closed under the action of that group, we see that the singular locus of $\affineClosureofBasicAffineSpace$ contains the cone point for such $G$ since any nonempty closed $\mathbb{G}_m$-invariant subscheme of $\affineClosureofBasicAffineSpace$ contains the cone point. Therefore $\text{dim}(T_0(\affineClosureofBasicAffineSpace)) > d := \text{dim}(G/N)$, and, since $\overline{z}\circ \projectionFromAffineClosureofCotangentBundleToAffineClosureofSpace = \text{id}$, $\projectionFromAffineClosureofCotangentBundleToAffineClosureofSpace$ induces a surjective map \begin{equation}\label{Surjection on Tangent Spaces}T_0(\affineClosureOfCotangentBundleofBasicAffineSpace) \to T_0(\affineClosureofBasicAffineSpace)\end{equation} on tangent spaces. 
    
    Note also that, by for example \cref{Cotangent Bundle of Smooth Locus is Fiber of affineClosureOfCotangentBundleofBasicAffineSpace Over Smooth Locus of G/N-bar}, the generic fiber of $\projectionFromAffineClosureofCotangentBundleToAffineClosureofSpace$ has dimension exactly $d$. Therefore, by upper semicontinuity of the fiber dimension, we see that the fiber $F$ of $\projectionFromAffineClosureofCotangentBundleToAffineClosureofSpace$ at the cone point of $\affineClosureOfCotangentBundleofBasicAffineSpace$ has dimension at least $d$. In particular, $T_0(F)$ has dimension at least $d$ and lies in the kernel of the map \labelcref{Surjection on Tangent Spaces}. Therefore by rank-nullity the dimension of $T_0(\affineClosureOfCotangentBundleofBasicAffineSpace)$ is larger than $d + d = \text{dim}(\affineClosureOfCotangentBundleofBasicAffineSpace)$, and so the cone point is singular. 
\end{proof}

On the other hand, the codimension of the singular locus of $\affineClosureOfCotangentBundleofBasicAffineSpace$ is at least four by \cref{Complement of Locus of Points on Which T Acts Freely for ARBITRARY REDUCTIVE GROUP Is at least Codimension Four}. Thus the $\mathbb{Q}$-factoriality of $\affineClosureOfCotangentBundleofBasicAffineSpace$ in \cref{affineClosureOfCotangentBundleofBasicAffineSpace is normal and Q-factorial for all G and is factorial when G is simply connected} allows us to use \cite[Corollary 1.3]{FuSymplecticResolutionsforNilpotentOrbits} to show the following, which generalizes a remark of \cite[Section 1.3]{GinzburgKazhdanDifferentialOperatorsOnBasicAffineSpaceandtheGelfandGraevAction} for $G = \SL_3$ to all types:

\begin{Corollary}
    If $G$ is semisimple and not a product of copies of $\SL_2$ then the variety $\affineClosureOfCotangentBundleofBasicAffineSpace$ does not admit a symplectic resolution.
\end{Corollary}

The fact that $\affineClosureOfCotangentBundleofBasicAffineSpace$ admits conical symplectic singularities for semisimple $G$ implies that is a natural object in the study of symplectic duality. The following result determines properties of the conjectural symplectic dual to $\affineClosureOfCotangentBundleofBasicAffineSpace$ and should be compared to the expectations of \cite[Section 8]{DancerHanayKirwanSymplecticDualityandImplosions}:

\begin{Corollary}\label{There are no nontrivial flat Poisson deformations of ringOfFunctionsForCOTANGENTBUNDLEOfBasicAffineSpace} There are no nontrivial flat Poisson deformations of $\affineClosureOfCotangentBundleofBasicAffineSpace$.
\end{Corollary}

\begin{proof}
        It suffices to show that the vector space $HP^2(\affineClosureOfCotangentBundleofBasicAffineSpace)$ is zero, see, for example, \cite{GinzburgKaledinPoissonDeformationsofSymplecticQuotientSingularities}, \cite{NamikawaPoissonDeformationsofAffineSymplecticVarieties}, \cite{NamikawaFlopsandPoissonDeformationsofSymplecticVarieties}. In fact, $HP^2(\mathcal{Y})$ vanishes for any normal affine variety $\mathcal{Y}$ with terminal symplectic singularities and finite class group. We give the proof of this well known result here for completeness.
    
    Since $\mathcal{Y}$ has terminal symplectic singularities, we have that $HP^2(\mathcal{Y}) \cong H^2(Y, \mathbb{C})$, where $Y$ denotes the smooth locus of $\mathcal{Y}$ \cite{NamikawaPoissonDeformationsofAffineSymplecticVarieties}. 
    Now \cite[Lemma 4.4.6]{LosevMasonBrownMatvieievskyiUnipotentIdealsandHarishChandraBimodules} shows that the first Chern class gives an isomorphism $\text{Pic}(Y) \otimes_{\mathbb{Z}} \mathbb{Q} \xrightarrow{\sim} H^2(Y, \mathbb{Q})$. Therefore, since $\text{Pic}(Y)$ is finite (since it is a subgroup of the class group of $\mathcal{Y}$) we see that $H^2(Y, \mathbb{C}) \cong H^2(Y, \mathbb{Q}) \otimes_{\mathbb{Q}} \mathbb{C} \cong 0$ as desired. 
\end{proof}

\printbibliography
\end{document}